\documentclass{amsart}
\usepackage{amsthm, amssymb,mathrsfs,mathtools,color, enumerate,fullpage, tikz-cd}

\usepackage[pagebackref]{hyperref}
\usepackage[capitalize, noabbrev]{cleveref}

\let\mathbb\undefined
\usepackage{etoolbox}
\patchcmd{\part}{\bfseries}{\scshape\Large}{}{}

\numberwithin{section}{part}

\newtheorem{definition}{Definition}[section]
\newtheorem{lemma}[definition]{Lemma}
\newtheorem{proposition}[definition]{Proposition}
\newtheorem{corollary}[definition]{Corollary}
\newtheorem{assumption}[definition]{Assumption}
\newtheorem{theorem}[definition]{Theorem}

\newtheorem*{notation}{Notation}

\theoremstyle{remark}
\newtheorem{remark}[definition]{Remark}

\title{Eisenstein degeneration of Euler systems}

\newcommand{\into}{\hookrightarrow}

\newcommand{\ZZ}{\mathbf{Z}}
\newcommand{\QQ}{\mathbf{Q}}
\newcommand{\CC}{\mathbf{C}}
\newcommand{\Qp}{\QQ_p}
\newcommand{\Qpb}{\overline{\QQ}_p}
\newcommand{\Qb}{\overline{\QQ}}
\newcommand{\Zp}{\ZZ_p}

\newcommand{\cF}{\mathcal{F}}
\newcommand{\cO}{\mathcal{O}}
\newcommand{\cR}{\mathcal{R}}
\newcommand{\cH}{\mathcal{H}}
\newcommand{\cW}{\mathcal{W}}

\newcommand{\hf}{{\mathbf{f}}}
\newcommand{\hg}{{\mathbf{g}}}
\newcommand{\hh}{{\mathbf h}}

\newcommand{\hj}{{\mathbf j}}
\newcommand{\hk}{{\mathbf k}}

\newcommand{\fp}{\mathfrak{p}}

\newcommand{\Lp}{{\mathscr{L}_p}}

\newcommand{\Dcris}{\mathbf{D}_{\mathrm{cris}}}

\renewcommand{\ge}{\geqslant}
\renewcommand{\le}{\leqslant}

\DeclareMathOperator{\GL}{GL}
\DeclareMathOperator{\Ad}{Ad}

\DeclareMathOperator{\Fil}{Fil}
\DeclareMathOperator{\Frob}{Frob}
\DeclareMathOperator{\Gal}{Gal}
\DeclareMathOperator{\crit}{crit}
\DeclareMathOperator{\Hom}{Hom}

\DeclareMathOperator{\ac}{ac}
\DeclareMathOperator{\loc}{loc}
\newcommand{\eps}{\varepsilon}
\newcommand{\rig}{\mathrm{rig}}

\newcommand{\sD}{\mathscr{D}}
\newcommand{\sV}{\mathscr{V}}

\newcommand{\et}{\text{\textup{\'et}}} 

\begin{document}

\begin{abstract}
We discuss the theory of Coleman families interpolating critical-slope Eisenstein series. We apply it to study degeneration phenomena at the level of Euler systems. In particular, this allows us to prove relations between Kato elements, Beilinson--Flach classes and diagonal cycles, and also between Heegner cycles and elliptic units. We expect that this method could be extended to construct new instances of Euler systems.
\end{abstract}
\renewcommand{\urladdrname}{\itshape ORCID}

\author{David Loeffler}
\address[D.L.]{Mathematics Institute, University of Warwick, Coventry, UK}
\curraddr{UniDistance Suisse, Schinerstrasse 18, 3900 Brig, Switzerland}
\email{david.loeffler@unidistance.ch}
\urladdr{\href{https://orcid.org/0000-0001-9069-1877}{\upshape\path{0000-0001-9069-1877}}}

\author{\'Oscar Rivero}
\address[O.R.]{Mathematics Institute, University of Warwick, Coventry, UK}
\curraddr{Department of Mathematics, Universidade de Santiago de Compostela, Spain}
\email{oscar.rivero@usc.es}
\urladdr{\href{https://orcid.org/0000-0001-9348-3651}{\upshape\path{0000-0001-9348-3651}}}

\thanks{Supported by ERC Consolidator grant ``Shimura varieties and the BSD conjecture'' (D.L.) and Royal Society Newton International Fellowship NIF\textbackslash R1\textbackslash 202208 (O.R.)}


 \maketitle

 \setcounter{tocdepth}{1}
 \tableofcontents
 \section*{Introduction}

  \subsection*{Overview} In the beautiful survey article \cite{BCDDPR}, Bertolini et al.~described several families of global cohomology classes arising from modular curves, including the Gross--Kudla--Schoen diagonal-cycle classes for a triple product of modular forms, and the Euler systems of Beilinson--Flach and Beilinson--Kato elements for two and for one modular form. They noted that the latter two constructions formally behave, in many ways, as if they were a ``degenerate case'' of the Gross--Kudla--Schoen classes with one or more of the cusp forms replaced by Eisenstein series. However, while this formal resemblance has proved very informative as a heuristic to guide the development of the theory, making it into a rigorous mathematical statement is not straightforward: since the Gross--Kudla--Schoen cycle extends to the triple product of the compactified modular curves $X_1(N)$, the projection of this cycle to a non-cuspidal Hecke eigenspace is zero.

  In this work, we develop an approach which allows us to make this ``Eisenstein degeneration'' of Euler systems into a rigorous theory. Our approach is based on two ingredients: the theory developed in \cite{loefflerzerbes16} to study the variation of Euler systems in Coleman families; and the existence of \emph{critical-slope Eisenstein series}, which are Eisenstein series arising as specialisations of Coleman families which are generically cuspidal.

  Using this method, we show the following: if $\hf$ is a Coleman family passing through a critical-slope Eisenstein series $f_\beta$, and $\hg, \hh$ are (cuspidal) Coleman families, then the process of ``specialisation at $f_\beta$'' relates Euler systems in the following way:
  \begin{enumerate}[(i)]
  \item The Beilinson--Flach classes for $\hf \times \hg$ go to a multiple of the Beilinson--Kato classes for $\hg$. This setting is studied in Section \ref{sect:BFdef}.
  \item The triple-product classes for $\hf \times \hg \times \hh$ are sent to the Beilinson--Flach classes for $\hg \times \hh$. This scenario is presented in Section \ref{sect:diag-def}.
  \item The Heegner classes for $\hf \times \lambda$, where $\lambda$ is a (suitable) Gr\"ossencharacter of an imaginary quadratic field, go to the elliptic-unit classes for $\lambda$. This is the content of the final section \ref{sect:heeg-def}.
  \end{enumerate}

 More precisely, in each case, we show that the image of the Euler system for the ``larger'' family is an Euler system class for the ``smaller'' family, multiplied by an additional, purely local ``logarithm'' term (and also by an extra $p$-adic $L$-function factor in case (ii), which can be naturally interpreted in terms of the Artin formalism for $L$-functions).

 Besides the intrinsic interest in relating these natural and important objects to each other, we hope that the techniques that we develop here may be of use in constructing new Euler systems, as in forthcoming work of Barrera et al.~discussed below. We hope to pursue this further in a future work.

We have also included an extensive section (Part B) devoted to recalling the main points in the theory for each of the Euler systems discussed in this note. We hope that this could help to reconcile the notations used in different papers, and could help as a guide to the less experienced reader willing to gain expertise in the area. Some of the results we present have been proved under certain simplifying assumptions, like considering tame level 1 for the degeneration of diagonal cycles to Beilinson--Flach classes; or restricting to class number 1 when discussing how to recover elliptic units beginning with Heegner points. Our purpose was illustrating our method without dealing with certain technical difficulties, but of course most of these hypotheses can be removed with some extra work.

\subsection*{Relation to other work} There are a number of prior works studying the specialisation of families of $p$-adic $L$-functions and/or Euler systems at points of the eigencurve corresponding to \emph{cuspidal} points of critical slope; see for instance \cite{loefflerzerbes12} and rather more recently \cite{buyukboduk-pollack-sasaki}.

In a slightly different direction, one can consider families of cusp forms degenerating to \emph{weight one} Eisenstein series, which exist when the Eisenstein series is non-$p$-regular. This approach has the advantage that the resulting families are ordinary, but on the other hand, the eigencurve is not smooth (or even Gorenstein) at such points; its local geometry has been studied in detail by Pozzi \cite{pozzi-thesis}. A forthcoming work of Barrera, Cauchi, Molina and Rotger will use this approach, applied to diagonal cycles on triple products of quaternionic Shimura curves, in order to define global Galois cohomology classes associated to products of one or two Hilbert modular forms. (This could potentially compensate for the fact that the original construction of the Beilinson--Flach and Beilinson--Kato Euler systems does not generalise to the Hilbert case, owing to the lack of suitable Eisenstein classes.) It will be interesting to explore the relation between their techniques and ours once their manuscript becomes available.

Finally, a recent work of Bertolini, Darmon and Venerucci \cite{BDV} proves a striking conjecture of Perrin-Riou relying on a comparison between Beilinson--Flach elements associated to weight 1 Eisenstein series and Beilinson--Kato elements, showing the strength of this kind of techniques.

\subsection*{Acknowledgements} The authors would like to thank Massimo Bertolini, Kaz\^im B\"uy\"ukboduk, Marco Seveso, Rodolfo Venerucci, and Sarah Zerbes for informative conversations related to this work, and Victor Rotger for his feedback on an earlier draft. We also thank the anonymous referees for a very careful reading of the text, whose comments notably contributed to improve the exposition of the article.
\clearpage

\part{Critical-slope Eisenstein series and their Galois modules}

 \section{Notations}

  We introduce notation for Eisenstein series, following the conventions of \cite{bellaichedasgupta15}. Let $\psi, \tau$ be two primitive Dirichlet characters of conductors $N_1, N_2$, and let $N = N_1 N_2$. Let $r \ge 0$ be such\footnote{We use $r$ for the weight of our Eisenstein series, rather than the more conventional $k$, since $k$ will later be the weight of a generic form in a family passing through the Eisenstein point (so $k$ may or may not be equal to $r$).} that $\psi(-1)\tau(-1) = (-1)^r$. If $r = 0$, assume $\psi$ and $\tau$ are not both trivial.

  \begin{definition}\label{def:Eis}
   We write $f = E_{r+2}(\psi, \tau)$ for the (classical) modular form of level $N$ and weight $r+2$ given by
   \[
    E_{r+2}(\psi, \tau) =  (*) + \sum_{n \ge 1} q^n \sum_{\substack{n = d_1 d_2 \\ (N_1, d_1) = (N_2, d_2) = 1}} \psi(d_1) \tau(d_2) d_2^{r+1},
   \]
   where $(*)$ is the appropriate constant defined e.g in \cite[\S1]{bellaichedasgupta15}. In particular we have $a_\ell(f) = \psi(\ell) + \ell^{r+1} \tau(\ell)$ for primes $\ell \nmid N$.
  \end{definition}

  We have two eigenforms of level $\Gamma_1(N) \cap \Gamma_0(p)$ associated to $E_{r+2}(\psi, \tau)$: one ordinary and one critical-slope, with eigenvalues
  \[ \alpha \coloneqq \psi(p) \qquad\text{and}\qquad \beta \coloneqq p^{r+1} \tau(p)\]
  respectively. We denote these two eigenforms by $f_\alpha$ and $f_\beta$ respectively. Observe that the role of the characters $\psi$ and $\tau$ is not symmetric, and when studying the corresponding Galois representations we will see that $\tau$ is always twisted by a suitable power of the cyclotomic character.

  We shall want to study the $p$-adic Galois representations attached to these forms, for a prime $p \nmid N$. It will be helpful to make the following definition:

  \begin{definition}
   We say the Eisenstein series $f = E_{r+2}(\psi, \tau)$ is \emph{$p$-decent} if one of the following conditions holds:
   \begin{itemize}
    \item $r > 0$;
    \item $r = 0$, and for every prime $\ell \mid Np$, either the conductor of $\psi / \tau$ is divisible by $\ell$, or $(\psi/\tau)(\ell) \ne 1$.
   \end{itemize}
  \end{definition}

  \begin{remark}
   This is slightly stronger than Bella\"iche's definition of ``decent'' \cite[Definition 1]{bellaiche11a}, since we also make the assumption at $p$.
  \end{remark}

 \section{Galois representations}
  \label{sect:galrep}

  We first fix notation for Galois characters. Let $p$ be a prime (which will be fixed for the duration of this paper).

  \begin{notation} \
   \begin{itemize}
    \item We write $\Frob_\ell$ for an arithmetic Frobenius element at $\ell$ in $\operatorname{Gal}(\Qb / \QQ)$; this depends, of course, on the choice of an embedding of $\Qb$ into $\Qb_{\ell}$, and is well-defined modulo the inertia subgroup for this embedding.

    \item The $p$-adic cyclotomic character $\operatorname{Gal}(\Qb / \QQ) \to \Qp^\times$ will be denoted by $\eps$, so that $\eps(\Frob_\ell) = \ell$ for $\ell \ne p$.

    \item If $\chi$ is a Dirichlet character, then we interpret $\chi$ as a character $\Gal(\Qb / \QQ) \to \CC^\times$ unramified at all primes $\ell$ not dividing the conductor, and mapping $\Frob_\ell^{-1}$ to $\chi(\ell)$ for all such primes $\ell$.
   \end{itemize}
  \end{notation}

  Now let $\chi, \psi$ be Dirichlet characters, and $r \ge 0$ an integer, as in \cref{def:Eis}; and choose an embedding $\QQ(\psi, \tau) \hookrightarrow L$, where $\QQ(\psi, \tau) \subset \CC$ is the finite extension of $\QQ$ generated by the values of $\psi$ and $\tau$, and $L$ is a finite extension of $\Qp$. Hence we may regard the $q$-expansion coefficients $a_n(f)$ of $f = E_{r+2}(\psi, \tau)$ as elements of $L$.

  \begin{theorem}[Soul\'e]
   \label{thm:soule}
   If $f$ is a $p$-decent Eisenstein series, there are exactly three isomorphism classes of continuous Galois representations $\rho: G_\QQ \to \GL_2(L)$ which are unramified at primes $\ell \nmid Np$ and satisfy $\operatorname{tr} \rho(\operatorname{Frob}_\ell^{-1}) = a_\ell(f)$. These are as follows:
   \begin{enumerate}
    \item The semisimple representation $\psi \oplus \tau \eps^{-1-r}$.

    \item Exactly one non-split representation having $\tau \eps^{-1-r}$ as a subrepresentation. This representation splits locally at $\ell$ for every $\ell \ne p$, but does not split at $p$, and is not crystalline (or even de Rham).

    \item Exactly one non-split representation having $\psi$ as a subrepresentation. This representation splits locally at $\ell$ for every $\ell \ne p$, and is crystalline at $p$.
   \end{enumerate}
  \end{theorem}

  This follows from cases of the Bloch--Kato conjecture due to Soul\'e; see \cite[\S 5.1]{bellaichechenevier06} for the statement in this form. If $f$ is not $p$-decent, there will be extra representations in case (3), but we can repair the statement by only considering representations which are assumed to be unramified (resp.~crystalline) at each prime $\ell \ne p$ (resp.~at $\ell = p$) where the decency hypothesis fails.

  \begin{remark}
   The dual of the representation in (3) is the one that Bella\"iche describes as the ``preferred representation'' associated to $f$ \cite[Lemma 2.10]{bellaiche11a}.
  \end{remark}

 \section{Classical cohomology}

  For an $L$-vector space $M$ with an action of the Hecke algebra of level $N$, we let $M[T=f]$ signify the maximal subspace of $M$ on which the Hecke operators $T(\ell), U(\ell)$ act as $a_\ell(f)$. Dually, we let $M[T' = f]$ for the corresponding eigenspace for the dual Hecke operators $T'(\ell)$, $U'(\ell)$.

  Let $Y_1(N)$ denote the modular curve classifying\footnote{Throughout this discussion we shall assume $N \ge 4$, so that $Y_1(N)$ exists as a fine moduli space. The remaining cases can be dealt with by the usual trick of passing to the moduli space of elliptic curves with a point of order $N$ and an auxiliary full level $R$ structure, for some $R \ge 3$ coprime to $pN$, and taking invariants under the action of $\GL_2(\ZZ/ R)$ on the cohomology.} elliptic curves with a point of order $N$ (identified with a quotient of the upper half-plane via $\tau \leftrightarrow \left(\mathbf{C} / (\ZZ + \ZZ\tau), \tfrac{1}{N}\right)$, so the cusp $\infty$ is not defined over $\QQ$). We let $\overline{Y_1(N)}$ denote the modular curve over $\Qb$.

  \subsection{Etale cohomology}

   Let $\mathscr{H}_{\Qp}$ be the relative Tate module of the universal elliptic curve over $Y_1(N)$, and $\mathscr{H}_{\Qp}^\vee$ for its dual (the relative \'etale $H^1$). We write $\mathscr{V}_r = \operatorname{Sym}^r\left(\mathscr{H}_{\Qp}^\vee\right) \otimes_{\Qp} L$.

   \begin{proposition}
    Let $f = E_{r + 2}(\psi, \tau)$ as above. Then the eigenspaces
    \[ H^1_{c}(\overline{Y_1(N)}, \mathscr{V}_r)[T = f] \quad\text{and}\quad H^1(\overline{Y_1(N)}, \mathscr{V}_r)[T = f] \]
    are both one-dimensional over $L$, but the natural map between them is 0. The group $\Gal(\overline{\QQ} / \QQ)$ acts on the former by $\psi$, and on the latter by $\tau \varepsilon^{-1-r}$.
   \end{proposition}

   Note that $H^1_{c}(\overline{Y_1(N)}, \mathscr{V}_r)$ is exactly the space of \emph{modular symbols} $\operatorname{Symb}_{\Gamma_1(N)}(\mathscr{V}_r)$. We endow these spaces with Galois actions following the normalisations of \cite{KLZ17}.

   \begin{remark}\label{rmk:pstab}
    If we work instead with the modular curve $Y$ of level $\Gamma=\Gamma_1(N) \cap \Gamma_0(p)$, and use either of the two $p$-stabilised eigenforms $f_{\alpha}, f_\beta$, then the resulting spaces are isomorphic to those at level $N$ via degeneracy maps, so the Galois actions are the same as above. That is, for $? = \alpha$ or $\beta$, the spaces
      \[
       H^1_{c}(\overline{Y}, \mathscr{V}_r)[T = f_?] \quad\text{and}\quad H^1(\overline{Y}, \mathscr{V}_r)[T = f_?]
      \]
      are $1$-dimensional, and isomorphic as Galois representations to $\psi$ and $\tau \eps^{-1-r}$ respectively.
   \end{remark}

  \subsection{De Rham cohomology}\label{sect:dRcoh}

   We have similar results for de Rham cohomology: the eigenspaces
   \[ H^1_{\mathrm{dR}, c}(Y_1(N), \mathscr{V}_r)[T = f] \quad\text{and}\quad H^1_{\mathrm{dR}}(Y_1(N), \mathscr{V}_r)[T = f]\]
   are both 1-dimensional over $L$, but the former has its grading concentrated in degree 0, and the latter in degree $r+1$. Using the BGG complex (with and without compact supports, see \cite[Prop. 2.4.1]{LSZ-asai}), we can define classical coherent-cohomology eigenclasses
   \[
    \eta_f \in H^1\left(X_1(N)_{\overline{\QQ}}, \omega^{-r}(-\mathrm{cusps})\right)[T = f]\quad\text{and}\quad \omega_f \in H^1\left(X_1(N)_{\overline{\QQ}}, \omega^{r+2}\right)[T = f]
   \]
   whose images span the $f$-eigenspaces in $H^1_{\mathrm{dR}, c}$ and $H^1_{\mathrm{dR}}$ respectively; these are characterised, as usual, by the requirement that $\omega_f$ be the image of the differential form associated to $f$, and $\eta_f$ pair to 1 with $\omega_{f^*}$ under Serre duality. Here, $f^*$ is the eigenform conjugate to $f$.

   In general $\eta_f$ and $\omega_f$ will not be defined over $L$, only over $L(\mu_N)$. We can correct this by multiplying them by a suitable Gauss sum, to make them $L$-rational, giving modified classes $\tilde{\eta}_f$, $\tilde{\omega}_f$ spanning the corresponding spaces over $L$, exactly as in the cuspidal case considered in \S 6.1 of \cite{KLZ20}. However, for our purposes it is simpler to assume that $L$ is large enough that it contains an $N$-th root of unity (so the Gauss sum is in $L$ anyway), in which case $\eta_f$ and $\omega_f$ are $L$-rational. We shall assume that $L$ satisfies this condition henceforth.


 \section{Overconvergent cohomology}

  Let $\Gamma = \Gamma_1(N) \cap \Gamma_0(p)$, and $Y = Y(\Gamma)$, equipped with a model over $\QQ$ compatibly with the one above for $Y_1(N)$ (so $Y$ is the modular curve denoted $Y(1, N(p))$ in \cite{kato04}). For each integer $r \ge 0$ we have an \'etale sheaf $\sD_r$ on $Y$, corresponding to the representation of $\Gamma$ on the dual space of the Tate algebra in one variable $z$, with the action of $\Gamma$ twisted by $(a + cz)^r$. This also makes sense for $r < 0$ (and indeed for any locally analytic character of $\Zp^\times$).

  \begin{remark}
   This sheaf is the sheaf denoted $\mathcal{D}_{r, m}(\mathscr{H}_0)(-r)$ in \cite{loefflerzerbes16}, with $m$ an auxiliary parameter (radius of analyticity); we take $m = 0$ and drop it from the notations.
  \end{remark}

  These sheaves are normally considered as topological (Betti) sheaves, but they can be promoted to \'etale sheaves on the canonical model of $Y$ as a $\ZZ[1/Np]$-scheme; cf.~\cite{andreattaiovitastevens,loefflerzerbes16}. Note that the Hecke operators away from $p$ act on the cohomology of $\sD_r$, as does the operator $U(p)$; but $U'(p)$ does \emph{not} act on this sheaf.

  As in \cite[\S3.2]{bellaiche11a}, we have a specialisation morphism $\sD_r \xrightarrow{\rho_r} \mathscr{V}_r$, which fits into an exact sequence of sheaves on $Y$,
  \[ 0 \longrightarrow \sD_{-2-r}(-1-r) \xrightarrow{\theta^{r + 1}} \sD_{r} \xrightarrow{\rho_r} \mathscr{V}_r \longrightarrow 0, \]
  where $\rho_r$ is the natural specialisation map, $(-1-r)$ denotes a Tate twist, and $\theta^{r + 1}$ is the dual of $(r + 1)$-fold differentiation.

  \subsection{Compact support}

   A theorem due to Stevens (see \cite[Lemma 5.2]{pollackstevens13}) shows that $H^2_c(\overline{Y}, \sD_{-2-r}) = 0$, and of course $H^0_c(\overline{Y}, \mathscr{V}_r)$ is also zero, so we obtain a short exact sequence of compactly-supported \'etale cohomology spaces
   \begin{equation}
    \label{eq:sescoho}
    0 \longrightarrow H^1_c(\overline{Y}, \sD_{-2-r}(-1-r)) \xrightarrow{\theta^{r + 1}} H^1_c(\overline{Y},\sD_{r}) \xrightarrow{\rho_r} H^1_c(\overline{Y},\mathscr{V}_r) \longrightarrow 0.
   \end{equation}
   This is an exact sequence of \'etale sheaves on $\operatorname{Spec} \ZZ[1/Np]$ (equivalently, of representations of $\operatorname{Gal}(\Qb / \QQ)$ unramified outside $Np$). The map $\rho_r$ is compatible with the Hecke operators, while the map $\theta^{r+1}$ is Hecke-equivariant up to a twist: we have $T(\ell) \circ \theta^{r+1} = \ell^{r+1} \theta^{r+1} \circ T(\ell)$ for primes $\ell \nmid Np$, and similarly for $U(\ell)$ with $\ell \mid Np$.


   \begin{proposition}[{\cite[Theorem 1]{bellaiche11a}}]
    Let $f_\beta = E_{r+2}^{\crit}(\psi, \tau)$ be the critical-slope $p$-stabilisation of a $p$-decent Eisenstein series $E_{r+2}(\psi, \tau)$, as before. Then for each sign $\pm$, the eigenspace
    \[ H^1_c(\overline{Y}, \sD_r)[T = f_\beta]^{\pm}\]
    where complex conjugation acts by $\pm 1$ is one-dimensional.
   \end{proposition}

   If $\varepsilon(f) = \psi(-1)$ as before, this shows that there is an ``extra'' eigenspace in the kernel of the classical specialisation map $\rho_r$, and complex conjugation acts on this space by $-\varepsilon(f)$.

   \begin{definition}
    We define
    \[ V^c(f_\beta) \coloneqq H^1_c(\overline{Y}, \sD_r)[T = f_\beta], \]
    which is a 2-dimensional representation of $\Gal(\Qb/\QQ)$.
   \end{definition}

   In practice we are interested in a more restrictive situation. For the following result, $M_{r + 2}^\dagger(\Gamma)_{(T = f_\beta)}$ stands for the generalised eigenspace of overconvergent modular forms associated to $f_{\beta}$.

   \begin{definition}[{\cite[Definition 2.14]{bellaiche11a}}]
    We say $f_\beta$ is \emph{non-critical} if the \emph{generalised} eigenspace of overconvergent modular forms $M_{r + 2}^\dagger(\Gamma)_{(T = f_\beta)}$ associated to $f_\beta$ is 1-dimensional.
   \end{definition}

   Note that here the notion of being \emph{critical} has to do, as in Bella\"iche, with the existence of a generalised eigenspace. In particular, as discussed in \cite[\S1.3, \S2.2.2]{bellaiche11a}, critical implies critical-slope, but not conversely. (Observe that in the notation $E_{r+2}^{\crit}(\psi, \tau)$, the superscript ``crit'' refers to this eigenform being critical-slope, but it is not necessarily critical.)

   \begin{theorem}[Bella\"iche--Chenevier, see {\cite[Remark 1.5]{bellaichedasgupta15}}]\label{equivalences}
    The following are equivalent:
    \begin{itemize}
     \item The form $f_\beta$ is non-critical.
     \item The localisation map
     \[ H^1_{\mathrm{f}}(\QQ, \psi \tau^{-1} \eps^{1+r}) \to H^1_{\mathrm{f}}(\Qp, \psi \tau^{-1} \eps^{1+r}) \]
     is non-zero.
     \item We have $L_p(\psi^{-1}\tau, r+1) \neq 0$, where $L_p$ denotes the Kubota--Leopoldt $p$-adic Dirichlet $L$-function.

     \item The Galois representation in case (3) of \cref{thm:soule} is not locally split at $p$.
    \end{itemize}
   \end{theorem}

   It is conjectured that these equivalent statements are always true (see the remarks \emph{loc.cit.}). The following partial result shows that they are true ``often'':

   \begin{proposition} \
    \begin{enumerate}[(i)]
     \item For any given $\psi$ and $\tau$, the Eisenstein series $E_{r+2}^{\crit}(\psi, \tau)$ is non-critical for all but finitely many integers $r \ge 0$ satisfying the parity condition $(-1)^r = \psi\tau(-1)$.

      \item If $r = 0$, and $\psi, \tau$ are such that the Eisenstein series $E_2^{\crit}(\psi, \tau)$ is $p$-decent, then this form is also non-critical.
    \end{enumerate}
   \end{proposition}

   \begin{proof}
    Part (i) follows from the second of the equivalent conditions of \cref{equivalences}: the $p$-adic $L$-function $L_p(\psi^{-1}\tau,s)$ is not identically 0 on the components of weight space determined by the parity condition, so it cannot vanish for infinitely many integers lying in these components.

    For part (ii) we use the fact that $L_p(\chi, 1) \ne 0$ for any non-trivial even Dirichlet character $\chi$, which is a consequence of the fact that Leopoldt's conjecture is known to hold for cyclotomic fields; see e.g.~\cite[Corollary 5.30]{washington97}.
   \end{proof}

   We assume henceforth that $f_\beta$ is non-critical.

   \begin{theorem}[{\cite[Theorem 4]{bellaiche11a}}] \
    \begin{itemize}
     \item Both generalised eigenspaces $H^1_c(\overline{Y}, \sD_r)_{(T = f_\beta)}^{\pm}$ are 1-dimensional over $L$.
     \item The $\varepsilon(f)$-eigenspace maps bijectively to its counterpart in $H^1_c(\overline{Y}, \mathscr{V}_r)$.
     \item The $-\varepsilon(f)$-eigenspace is isomorphic, via the $\theta^{r + 1}$ map, to the $[T = g]$ eigenspace in $H^1_c(\overline{Y}, \sD_{-2-r})^{\varepsilon(f)}$, where $g = E_{-r}^{\mathrm{ord}}(\tau, \psi)$ (note the reversal of the order of the characters) is the unique eigenform satisfying $\theta^{r + 1}(g) = f_\beta$.
    \end{itemize}
   \end{theorem}

   \begin{proposition}\label{prop:SEScompact}
    There is a short exact sequence of $L$-vector spaces
    \[ 0 \longrightarrow \tau \eps^{-1-r} \longrightarrow V^c(f_\beta) \longrightarrow \psi \longrightarrow 0.\]
   \end{proposition}

   \begin{proof}
    We know that the space $V^c(f_\beta)$ has a one-dimensional quotient isomorphic to $\psi$, by \eqref{eq:sescoho} and Remark \ref{rmk:pstab}. On the other hand, also by \eqref{eq:sescoho}, the kernel is isomorphic to the $(-1-r)$-th twist of the $H^1_c$ eigenspace associated to the non-classical ordinary $p$-adic Eisenstein series $g$ of weight $-r$. Since $f_\beta$ is non-critical, this space must be 1-dimensional. So it suffices to show that $\Gal(\Qb / \QQ)$ acts on the one-dimensional space $H^1_c(\overline{Y}, \sD_{-2-r})[T = g]$ via $\tau$.

    We now recall (see \cite[(2.3.10)]{ohta99} for example) that if $\mathcal{W}_{-2-r}$ denotes the component of weight space containing the integer $-2-r$, then there exists a $p$-adic family of Eisenstein series over $\mathcal{W}_{-2-r}$, whose specialisation in a classical weight $k \in \mathcal{W}_{-2-r} \cap \ZZ_{\ge 0}$ is the classical ordinary Eisenstein series $E_{k+2}^{\mathrm{ord}}(\tau, \psi)$, and whose specialisation at $-2-r$ is $g = E_{-r}^{\mathrm{ord}}(\tau, \psi)$.

    From the control theorem for \'etale cohomology in (possibly non-cuspidal) Hida families proved in \cite{ohta99} (which is also valid, with the same proof, for compactly-supported cohomology), it follows that there is an open disc $U$ around $-2-r$ in weight space, and a finite-rank free $\cO(U)$-module with an action of $G_{\QQ}$, whose specialisation at a weight $\kappa \in U$ is the eigenspace $H^1_c(\overline{Y}, \sD_\kappa)[T = E_{\kappa+2}^{\mathrm{ord}}(\tau, \psi)]$. By applying Remark \ref{rmk:pstab} to the specialisation at each integer $k \ge 0$ in $U$ (and using the fact that such points are Zariski-dense in $U$), we see that this module must be of rank 1 over $\cO(U)$, and $G_{\QQ}$ acts on it by the character $\tau$. Specialising back to $\kappa =  -2-r$ we obtain the result.
   \end{proof}

   \begin{remark}
    This shows that the isomorphism class of $V^c(f_\beta)$ must be one of the isomorphism classes (1) or (2) in \cref{thm:soule}. We shall see shortly that it is in fact non-split, so it lies in the isomorphism class (2), and is not de Rham at $p$.
   \end{remark}


  \subsection{Non-compact supports}

   We now consider the non-compactly-supported space, continuing to assume $f_\beta$ is $p$-decent and non-critical.

   \begin{definition}
    We define
    \[ V(f_\beta) = H^1(\overline{Y}, \sD_r)[T = f_\beta].\]
   \end{definition}

   There is a canonical map $H^1_c(\overline{Y}, \sD_r) \to H^1(\overline{Y}, \sD_r)$; its kernel is the space of \emph{boundary modular symbols}. This gives a map of Galois representations $V^c(f_\beta) \to V(f_\beta)$.

   \begin{proposition}[{\cite[Remark 5.10]{bellaichedasgupta15}}]
    The intersection of $V^c(f_\beta)$ with the boundary symbols is 1-dimensional, and is exactly the $\tau \eps^{-1-r}$ subrepresentation.
   \end{proposition}

   Hence both eigenspaces $H^1(Y, \sD_r)[T = f_\beta]^{\pm}$ must have dimension at least 1. In fact these dimensions are both exactly 1, and are equal to the corresponding generalised eigenspaces, since Bellaiche's argument in \cite[Theorem 3.30]{bellaiche11a} works for non-compactly-supported cohomology also. We deduce that there is a short exact sequence of $L$-linear Galois representations
   \begin{equation}\label{ses-eis}
   0 \longrightarrow \psi \longrightarrow V(f_\beta) \longrightarrow \tau \eps^{-1-r} \longrightarrow 0,
   \end{equation}
   so $V^c(f_\beta)$ and $V(f_\beta)$ have isomorphic semi-simplifications; but the natural map between them is not an isomorphism, but rather identifies the one-dimensional $\psi$-isotypic quotient of $V^c(f_\beta)$ given by \cref{prop:SEScompact} with the one-dimensional $\psi$-isotypic subspace of $V(f_\beta)$ given by \eqref{ses-eis}.

   \begin{remark}
    As with $V^c(f_\beta)$ above, we have not yet determined whether $V(f_\beta)$ is a split or non-split extension; so it could be either of the isomorphism classes (1) or (3) of \cref{thm:soule}, but in either case it is de Rham at $p$. We shall see shortly that it is in fact the non-split extension (3).
   \end{remark}

 \section{Duality and Atkin--Lehner}

  As in \cite{loefflerzerbes16}, there is a second family of sheaves $\sD'_r$ (denoted by the more verbose notation $\mathcal{D}_{r, m}(\mathscr{H}'_0)$ in \emph{op.cit.}), which also interpolates the standard finite-dimensional sheaves, but has an action of $U'(p)$ rather than $U(p)$. It is the sheaf $\sD'_r$ which is used for interpolating Euler system classes in families.

  \begin{remark}
   There is a canonical $L$-linear pairing between $\sD_r$ and $\sD'_r$, landing in the constant sheaf $L$; but it is not perfect in general (indeed, we shall see shortly that it does not induce a perfect duality on the fibre at a critical-slope Eisenstein point).
  \end{remark}

  The two sheaves are interchanged by the Atkin--Lehner involution, modulo a twist of the Galois action; so the above structural results for $\sD$ carry over \emph{mutatis mutandis} to $\sD'$. So, for $r \in \ZZ_{\ge 0}$, we have exact sequences of sheaves on $Y$
  \[
   0 \longrightarrow \sD'_{-2-r}(1 + r) \xrightarrow{\theta^{r + 1}} \sD_r' \xrightarrow{\rho_r} \mathscr{V}_r^* \longrightarrow 0,
  \]
  where $\mathscr{V}_r^*$ is the linear dual of $\mathscr{V}_r$ (isomorphic to the Tate twist $\mathscr{V}_r(r)$).

  \begin{definition}
   \label{def-vfbeta}
   For $f_\beta = E_{r+2}^{\crit}(\psi, \tau)$ as above, we define
   \[ V(f_\beta)^* = H^1(\overline{Y}, \sD_r'(1))[T' = f_\beta], \]
   which is a 2-dimensional $L$-linear representation of $\Gal(\Qb/\QQ)$ fitting into an exact sequence
   \begin{equation}
    \label{ses-coh2}
     0 \to \tau^{-1} \eps^{1 + r} \to V(f_\beta)^* \to \psi^{-1} \to 0.
   \end{equation}
   We define $V^c(f_\beta)^*$ as the analogous space with compactly-supported rather than non-compactly-supported cohomology, so that
   \begin{equation}
    \label{ses-coh2c}
     0 \to \psi^{-1} \to V^c(f_\beta)^* \to \tau^{-1} \eps^{1 + r} \to 0,
   \end{equation}
   and there is a natural map $V^c(f_\beta)^* \to V(f_\beta)^*$ whose image is the $\tau^{-1}\eps^{1+r}$ subrepresentation of the latter.
  \end{definition}

  \begin{remark}\label{rmk:notperfect}
   We have therefore defined four Galois representations $V(f_\beta)$, $V(f_\beta)^*$, $V^c(f_\beta)$, and $V^c(f_\beta)^*$ associated to $f_\beta$, all of which are 2-dimensional, with natural 1-dimensional invariant subspaces; and we have identified the Galois actions on their graded pieces.

   We shall see in the next section that all four representations are \emph{non-split} extensions. From this, it follows that $V(f_\beta)^*$ is isomorphic to the dual of $V(f_\beta)$ (with both being de Rham at $p$), while $V^c(f_\beta)^*$ is isomorphic to the dual of $V^c(f_\beta)$ (with both being non-de Rham at $p$), justifying our choice of notations. However, the pairing giving this duality is not quite the natural Poincar\'e duality pairing. What we obtain from a naive application of Poincar\'e duality is pairings
   \begin{equation}
    \label{eq:notperfect}
    V^c(f_\beta) \times V(f_\beta)^* \to L, \qquad V(f_\beta) \times V^c(f_\beta)^* \to L
   \end{equation}
   which are \emph{not} perfect, but rather factor through the 1-dimensional classical quotients of these 2-dimensional representations.
  \end{remark}

 \section{P-adic families}

  \subsection{Pseudocharacters}

   The form $f_\beta$ is an overconvergent cuspidal eigenform of finite slope, so it defines a point on the Coleman--Mazur--Buzzard cuspidal eigencurve $\mathcal{C}^0$ of tame level $N$. Moreover, our assumptions that $f$ be decent and non-critical imply that $\mathcal{C}^0$ is smooth at $f_\beta$ and locally \'etale over weight space (see Proposition 2.11 of \cite{bellaiche11a}); so we may choose a closed disc $U$ around $f_\beta$ which maps isomorphically to its image in weight space.

   We let $\hf$ be the universal eigenform over $U$, which is an overconvergent cuspidal eigenform with coefficients in $\cO(U)$, and weight $\hk + 2$ where $\hk: \Zp^\times \to \cO(U)^\times$ is the canonical character; by definition, the specialisation of $\hf$ at $r \in U$ is $f_\beta$. Let us write $\beta_{\hf} \in \cO(U)^\times$ for the $U(p)$-eigenvalue of $\hf$, so that $\beta_{\hf}(r) = \beta = p^{r+1} \tau(p)$.

   There is a canonical Galois pseudo-character $t_{\hf}: G_\QQ \to \cO(U)$ satisfying
   \[ t_{\hf}(\Frob_\ell^{-1}) = a_\ell(\hf)\]
   for good primes $\ell$. If $\mathfrak{m}$ is the maximal ideal of $\cO(U)$ corresponding to $r$, then $t_{\hf}$ is reducible modulo $\mathfrak{m}$ and its reduction is the sum of two distinct characters. Up to shrinking $U$, we have the following result.

   \begin{theorem}[{\cite[Theorem 1]{bellaichechenevier06}}]
    \label{thm:redideal}
    The reducibility ideal of the pseudo-character $t_{\hf}$ is exactly $\mathfrak{m}$.\qed
   \end{theorem}

   (More precisely, this is proved in \cite{bellaichechenevier06} for $N = 1$. However, as noted in Remark 2.9 of \cite{bellaiche11a}, this is purely because a construction of the eigencurve of tame level $> 1$ was not available in the literature at the time \cite{bellaichechenevier06} was written, and all of the arguments of \emph{op.cit.}~go over without change to $p$-decent Eisenstein series of higher level.)

  \subsection{Families of representations}

   We now define two specific $\cO(U)$-linear Galois representations which both have traces equal to $t_{\hf}$, and two more whose traces are the dual $t_{\mathbf{f}}^*$. As shown in \cite[\S 4]{loefflerzerbes16}, there exist sheaves $\sD_U$ and $\sD'_U$ of $\cO(U)$-modules on $Y$, whose specialisations at any $\kappa \in U$ are the sheaves $\sD_\kappa$ and $\sD_\kappa'$ considered above.

   \begin{definition}
    \label{rep-hida}
    We define
    \begin{align*}
     V(\hf) &\coloneqq H^1_{\et}\left(\overline{Y}, \sD_U\right)[T = \hf] &
     V^c(\hf) &\coloneqq H^1_{\et,c}\left(\overline{Y}, \sD_U\right)[T = \hf] \\
     V(\hf)^* &\coloneqq H^1_{\et}\left(\overline{Y}, \sD'_U(1)\right)[T' = \hf] &
     V^c(\hf)^* &\coloneqq H^1_{\et,c}\left(\overline{Y}, \sD'_U(1)\right)[T' = \hf]
    \end{align*}
   \end{definition}

   It follows from the results of \cite{andreattaiovitastevens,loefflerzerbes16} that if $f_\beta$ is non-critical, then (up to possibly shrinking $U$) the module $V(\hf)$ is free of rank 2 over $\cO(U)$ and a direct summand of $H^1_{\et}(\overline{Y}, \sD_U)$, so we have a natural Hecke-equivariant map
   \[
    \operatorname{Pr}_{\hf}: H^1_{\et}(\overline{Y}, \sD_U) \to V(\hf);
   \]
   and $V(\hf)$ carries a $\cO(U)$-linear Galois representation whose trace is $t_{\hf}$. Since $t_{\hf}$ has maximal reducibility ideal, it follows that the fibre of $V(\hf)$ at $k=r$, which is exactly $V(f_\beta)$, must be a non-split extension. The same applies \emph{mutatis mutandis} to the other three modules, showing that all four of $V(f_\beta), V^c(f_\beta), V(f_\beta)^*$ and $V^c(f_\beta)^*$ are non-split extensions.

  \begin{corollary} \
   \begin{enumerate}[(i)]
    \item The isomorphism class of the representation $V(f_\beta)$ is the unique non-split, de Rham extension in case (3) of \cref{thm:soule}. The isomophism class of $V^c(f_\beta)$ is the non-de Rham extension in case (2) of the theorem.
    \item $V(f_\beta)^*$ is isomorphic to the dual of $V(f_\beta)$, and $V^c(f_\beta)^*$ to the dual of $V^c(f_\beta)$.
   \end{enumerate}
  \end{corollary}

 \subsection{Comparison of $\cO(U)$-lattices}\label{inclusions}

  There is a natural `forget supports' map $V^c(\hf) \to V(\hf)$, which is $\cO(U)$-linear. Since this map specialises to an isomorphism at any non-critical-slope classical point, it must be injective, with torsion cokernel; thus we may regard both $V^c(\hf)$ and $V(\hf)$ as $\cO(U)$-lattices in $V(\hf) \otimes_{\cO(U)} \operatorname{Frac} \cO(U)$. Shrinking our disc $U$ further if necessary, we may assume that the cokernel is supported at $\mathfrak{m}$.

  Since the map $V^c(f_\beta) \to V(f_\beta)$ is not the zero map, $V^c(\hf)$ is not contained in $\mathfrak{m} \cdot V(\hf)$. Thus we may find a basis $(e_1, e_2)$ of $V(\hf)$, and an integer $r \ge 1$, such that $(e_1, X^r e_2)$ is a basis of $V^c(\hf)$, where $X$ is a uniformizer of $\mathfrak{m}$. From \cref{thm:redideal}, we must in fact have $r = 1$. Thus $X V(\hf) \subset V^c(\hf)$, and we have a chain of inclusions
  \[ \dots \supset \tfrac{1}{X} V^c(\hf) \supset V(\hf) \supset V^c(\hf) \supset X V(\hf) \supset \dots \]
  with all of the successive quotients l-dimensional over $L$, and alternately equal to either $\psi$ or $\tau \varepsilon^{-1-r}$ as Galois modules. Similarly, we have a chain
  \begin{equation}\label{chain}
   \dots \supset \tfrac{1}{X} V^c(\hf)^* \supset V(\hf)^* \supset V^c(\hf)^* \supset X V(\hf)^* \supset \dots
  \end{equation}
  with quotients alternately isomorphic to either $\psi^{-1}$ or $\tau^{-1} \varepsilon^{1+r}$.

 \subsection{Duality}

  We recall the construction of the ``Ohta pairing'' from \cite[\S 4.3]{loefflerzerbes16}, which is a $\cO(U)$-bilinear pairing on  $\sD_U \times \sD'_U$, taking values in the constant sheaf $\cO_U$. This gives rise to a pairing of Galois representations $\{-, -\} : V^c(\hf) \times V(\hf)^* \to \cO(U)$, interpolating the Poincar\'e duality pairings on the classical specialisations.

  This pairing is not perfect, since the induced pairing on the fibre at $\mathfrak{m}$ is the pairing $V^c(f_\beta) \times V(f_\beta)^* \to L$ (which we have seen is non-perfect). However, since the Poincar\'e duality pairings on non-critical-slope, cuspidal specialisations of $\hf$ are perfect, the map $V^c(\hf) \to \Hom_{\cO(U)}(V(\hf)^*, \cO(U))$ induced by the Ohta pairing must be injective, with torsion cokernel. Shrinking $U$ if needed, we may suppose the cokernel is supported at $\mathfrak{m}$.

  \begin{proposition}\label{prop:Ohta}
   The $\cO(U)$-submodule
   \[ \left\{ x \in V^c(\hf)[1/X] : \{x, y\} \in \cO(U)\  \forall y \in V(\hf)^* \right\}\]
   is equal to $V(\hf) \subset \tfrac{1}{X} V^c(\hf)$, and hence the Ohta pairing extends uniquely to a perfect $\cO(U)$-linear pairing $V(\hf) \times V(\hf)^* \to \cO(U)$.
  \end{proposition}

  \begin{proof}
   Let us write $W$ for the $\cO(U)$-lattice $\left\{ x \in V^c(\hf)[1/X] : \{x, y\} \in \cO(U)\  \forall y \in V(\hf)^* \right\}$. Evidently $W \supseteq V^c(\hf)$, and the quotient is $X$-torsion. The image of $V^c(\hf) / X V^c(\hf) \cong V^c(f_\beta)$ in $W / XW \cong \Hom_L(V(f_\beta)^*, L)$ is not zero: it is exactly the 1-dimensional classical quotient of $V^c(f_\beta)$. So the quotient $W / V^c(\hf)$ is isomorphic to $\cO(U) / X^n$, for some $n$.

   However, since both $W$ and $V^c(f_\beta)$ are $G_{\QQ}$-invariant $\cO(U)$-lattices in $V^c(\hf)[1/X]$, this implies that the action of Galois in a suitable $\cO(U)$-basis of $W$ is upper-triangular modulo $X^n$; and if $n \ge 2$, this contradicts the maximality of the reducibility ideal of the pseudocharacter $t_{\hf}$.

   Hence $W$ is a Galois-invariant sublattice of $\tfrac{1}{X} V^c(\hf)$ containing $V^c(\hf)$, with both containments strict. As $V^c(\hf) / X \cong V^c(f_\beta)$ has a unique Galois-invariant line, there is a unique such lattice, namely $V(\hf)$.
  \end{proof}

  \begin{remark}
   In particular, there is a uniquely determined perfect pairing $V(f_\beta) \times V(f_\beta)^* \to L$, refining the result above that these representations are abstractly dual to each other.
  \end{remark}

 \section{Nearly overconvergent cohomology}

  We summarise here some results from \cite[\S 5.2]{loefflerzerbes16} on ``nearly overconvergent \'etale cohomology''. The basic object of study is the cohomology of of the sheaves $\sD'_{U - j} \otimes \sV_j$, for $j \in \ZZ_{\ge 0}$, which we interpret as ``nearly overconvergent cohomology of degree $j$''. Here $U - j$ denotes the image of the disc $U$ under the map $\cW \to \cW, \kappa \mapsto \kappa - j$, so that for every integer $k \in U \cap \ZZ$ with $k \ge j$, the fibre of $\sD'_{U - j} \otimes \sV_j$ at $k$ surjects onto $\sV_{k - j} \otimes \sV_j$.

  \begin{remark}
   These modules are relevant to our present study because the Beilinson--Flach classes associated to Coleman families constructed in \emph{op.cit.} naturally land in these larger sheaves, rather than in the $\sD'_{U}$ sheaves themselves, as we shall recall in more detail below.
  \end{remark}

  As explained \emph{loc.cit.}, there is a natural map of sheaves
  \[ \beta_j^* : \sD'_U \to \sD'_{U - j} \otimes \sV_j, \]
  compatible with the Clebsch--Gordan maps $\sV_k \into \sV_{k - j} \otimes \sV_j$ for integers $k \ge j$; and there is a map the other way,
  \[ \delta_j^* : \sD'_{U - j} \otimes \sV_j \to \sD'_U,\]
  such that $\delta_j^* \circ \beta_j^*$ is multiplication by $\binom{\nabla}{j} \in \cO_U$. Both of these maps are $\cO(U)$-linear, and compatible with Hecke correspondences away from $p$, and with $U'(p)$.

  Since $\cO(U)$ is reduced and $\binom{\nabla}{j} \ne 0$, we may make sense of the map
  \[ \Pr{}^{[j]} \coloneqq \tfrac{1}{\binom{\nabla}{j}} \delta_j^* :\ \ H^1(\overline{Y}, \sD'_{U - j} \otimes \sV_j) \to \tfrac{1}{\binom{\nabla}{j}} H^1(\overline{Y}, \sD'_{U}), \]
  whose specialisation at any $k \ge j$ is a left inverse of $\beta_j^*$. The denominator has simple poles at all locally-algebraic characters of degree $k \in \{0, \dots, (j-1)\}$; but the residues at these poles are valued in the kernel of specialisation on $\sD'_{k}$, since the composite $\sD'_{k - j} \otimes \sV_j \xrightarrow{\delta_j^*} \sD'_k \to \sV_k$ is zero.

  We shall need this map in the case when $\hf$ is a Coleman family with one critical-slope Eisenstein fibre in weight $r$, and $j = r + 1$. Shrinking $U$ if necessary, we can assume that the specialisations at all points of $U$ which are locally-algebraic of degree $0, \dots, r$ are cuspidal and non-critical-slope, except possibly at $r$ itself. Thus $\Pr_{\hf}^{[r+1]}$ on cohomology takes values in $\tfrac{1}{X} V(\hf)^*$; but its residue maps trivially into $V(f_\beta)^*_{\mathrm{quo}}$, so in fact it factors through the slightly smaller module $\tfrac{1}{X} V^c(\hf)^* \supset V(\hf)^*$, where $X$ is a uniformizer at $r \in U$ (cf.~equation \ref{chain}). Thus we obtain a map of Galois representations over $\cO(U)$,
  \[ \operatorname{Pr}^{[r+1]}_{\hf} : H^1(\overline{Y}, \sD'_{U - (r + 1)} \otimes \sV_{r + 1}(1)) \to \tfrac{1}{X}V^c(\hf^*), \]
  whose composite with $\beta_j^*$ coincides with the natural projection map $\operatorname{Pr}_\hf:  H^1(\overline{Y}, \sD'_{U}(1)) \to V(\hf)^* \subset \tfrac{1}{X} V^c(\hf)^*$.

 \section{P-adic Hodge theory}

  We now investigate the restriction of the Galois modules constructed above to the decomposition group at $p$. We fix an embedding $\Qb \into \Qpb$, so we may regard $\Gal(\Qpb/\Qp)$ as a subgroup of $\Gal(\Qb/\QQ)$.
  \subsection{Notations}

   \begin{itemize}
    \item Let $\cR$ be the Robba ring over $\Qp$, and let $\cR_U = \cR \mathop{\hat\otimes} \cO_U$.
    \item Let $t \in \cR$ be the period for the cyclotomic character, so $\varphi(t) = pt$ and $\gamma(t) = \varepsilon(\gamma) t$.
    \item For $\delta \in \cO(U)^\times$, let $\cR_U(\delta)$ be the rank 1 $(\varphi, \Gamma)$-module over $\cR_U$ generated by an element $e$ which is $\Gamma$-invariant and satisfies $\varphi(e) = \delta e$.
    \item For $D$ a $(\varphi, \Gamma)$-module over $\cR$ (or $\cR_U$), write $\Dcris(D) = D[1/t]^{\Gamma}$, with its filtration $\Fil^n \Dcris(D) = (t^n D)^{\Gamma}$; this is compatible with the usual notations for $D = \Dcris(V)$, $V$ crystalline.

    \item Let $D(\hf)^* = \mathbf{D}^\dag_\rig(V(\hf)^*)$, and similarly $D^c(\hf)^*$.
   \end{itemize}

  \subsection{Triangulations for $D(\hf)^*$}

   We know that the space of homomorphisms
   \[ \cR_U(\frac{\beta_\hf}{\psi(p) \tau(p)})(1 + \hk) \to D(\hf)^*\]
   is a finitely-generated $\cO(U)$-module, by the main theorem of \cite{KPX}. Thus it must be free of rank 1, since it is clearly torsion-free, and there is a Zariski-dense set of $x \in \cO(U)$ where the fibre is 1-dimensional. Let $\cF^+ D(\hf)^*$ be the image of a generator of this map, and $\cF^-$ the quotient, so that we have a short exact sequence
   \[ 0 \to \cF^+ D(\hf)^* \to D(\hf)^* \to \cF^- D(\hf)^* \to 0.\]

   Over the punctured disc $U - \{ r \}$, the sub and quotient are both free of rank 1 and so the above sequence defines a triangulation of $D(\hf)^*$.

   \begin{proposition}
    If $f_\beta$ is non-critical, the submodule $\cF^+ D(f_\beta)^*$ is saturated; thus $\cF^-D(\hf)^*$ is free of rank 1 over $\cR_U$, and $\cF^{\pm} D(\hf)^*$ define a triangulation of $D(\hf)^*$ over $U$.
   \end{proposition}

   \begin{proof}
    This follows by a comparison of filtration degrees on $\Dcris$: if the submodule were not saturated, then the $\varphi = \psi(p)^{-1}$-eigenspace in $\Dcris(V(f_\beta)^*)$ would be contained in $\mathrm{Fil}^n$ for some $n > -1-r$, and hence necessarily in $\mathrm{Fil}^0$. This would force $V(f_\beta)^*$ to locally split at $p$ as a direct sum, which according to Theorem \ref{equivalences} contradicts the assumption that $f_\beta$ is non-critical.
   \end{proof}

   We thus have two exact sequences ``cutting across each other'', one arising from the triangulation (the horizontal one in the diagram), and one from the global reducibility of $V(f_\beta)^*$ (the vertical one):
   \[\begin{tikzcd}
    &&0 \dar\\
    &   & \mathbf{D}^{\dag}( V(f_\beta)^*_{\mathrm{sub}})\dar
    \arrow[rd, dashed, bend left, hook, "\text{cokernel $t^{r+1}$}" above right]
    \\
    0 \rar & \cF^+ D(f_\beta)^* \rar
    \arrow[rd, bend right, dashed, hook, "\text{cokernel $t^{r+1}$}" below left]
    & D(f_\beta)^*\dar  \rar & \cF^- D(f_\beta)^* \rar & 0\\
    &                     & \mathbf{D}^{\dag}( V(f_\beta)^*_{\mathrm{quo}})\dar\\
    &&0
   \end{tikzcd}\]
   There is an isomorphism $\mathbf{D}^{\dag}( V(f_\beta)^*_{\mathrm{quo}}) \cong \cR(\psi(p)^{-1})$, and the lower-left dotted arrow must, therefore, identify its source with the $(\varphi, \Gamma)$-stable submodule $t^{r+1} \cR(\psi(p)^{-1})$ of the target; and similarly for the upper right dotted arrow. This corresponds to the fact that the map
   \[ \Dcris\left( \cF^+ D(f_\beta)^* \right) \to \Dcris\left(  V(f_\beta)^*_{\mathrm{quo}} \right)\]
   is an isomorphism on the underlying $\varphi$-modules, but shifts the filtration degree by $r+1$.

  \subsection*{Triangulations for $D^c(\hf)^*$}

   For $D^c(\hf)^*$, the situation is a little different: in this case, the triangulation becomes singular in the fibre at $r$. More precisely, we may choose generators of the free rank 1 $\cO(U)$-modules
   \[
    \Hom_{(\varphi, \Gamma)}\Big( R_U(\beta_{\hf} / \psi\tau(p))(1 + \hk), D^c(\hf)^*\Big)\quad\text{and}\quad
    \Hom_{(\varphi, \Gamma)}\Big( D^c(\hf)^*, R_U(\beta_{\hf}^{-1}), \Big).
   \]
   Then we obtain a submodule $\cF^+ D^c(\hf)^*$ and a quotient $\cF^- D^c(\hf)^*$ which restrict to the triangulation away from weight $r$, and such that the maps
   \[ \cF^+ D^c(f_\beta)^* \to D^c(f_\beta)^* \quad\text{and}\quad D^c(f_\beta)^* \to \cF^- D^c(f_\beta)^*\]
   are nonzero, where $\cF^{\pm} D^c(f_\beta)^*$ are the fibres of $\cF^\pm D^c(\hf)^*$ at $r$. If we identify $V^c(\hf)^*$ with a submodule of $V(\hf)^*$, then we deduce that
   \[ \cF^+ D^c(\hf)^* = X \cdot \cF^+ D(\hf)^*, \qquad  \cF^- D^c(\hf)^* = \cF^- D(\hf)^*,\]
   where $X$ is a uniformizer at $r \in U$.

   However, these two maps do not define a triangulation, because the submodule $\cF^+ D^c(f_\beta)^*$ is not saturated: its image in $D^c(f_\beta)^*$ is exactly $t^{r+1} \cdot \mathbf{D}^{\dag}( V^c(f_\beta)^*_{\mathrm{sub}})$. Similarly (and in fact dually), the quotient map $D^c(f_\beta)^*\to\cF^- D^c(f_\beta)^*$ has image $t^{r+1} \cF^- D^c(f_\beta)^*$, which we can identify with $\mathbf{D}^{\dag}( V^c(f_\beta)^*_{\mathrm{quo}}) \cong \cR(p^{r+1}/\beta)(r+1)$. So we have an analogous ``cross'' as before, but the horizontal row is \emph{not exact}:
   \[
    \begin{tikzcd}
     &&0 \dar\\
     &   & \mathbf{D}^{\dag}( V^c(f_\beta)^*_{\mathrm{sub}})\dar
     \\
     0 \rar & \cF^+ D^c(f_\beta)^* \arrow[r,"!" description] \arrow[ru, dashed, bend left, hook, "\text{cokernel $t^{r+1}$}" above left]
     & D^c(f_\beta)^*\dar  \arrow[r,"!" description] & \cF^- D^c(f_\beta)^* \rar & 0\\
     && \mathbf{D}^{\dag}( V^c(f_\beta)^*_{\mathrm{quo}})\dar\arrow[ur, hook, dashed, bend right, "\text{cokernel $t^{r+1}$}" below right]\\
     &&0
    \end{tikzcd}
   \]
   The lower right arrow induces an isomorphism after inverting $t$, and hence is an isomorphism between $\Dcris$ modules (since $\Dcris(D) = D[1/t]^{\Gamma}$); but since the filtration on $\Dcris$ is given by $(t^n D)^{\Gamma}$, the isomorphism shifts the filtration degrees -- the filtration on $\Dcris(V^c(f_\beta)^*_{\mathrm{quo}})$ is concentrated in degree $-1-r$, but the filtration on $\Dcris\left( \cF^- D^c(f_\beta)^*\right)$ is concentrated in degree 0.

  \subsection{Crystalline periods}\label{subsec:periods}

   Essentially by construction, we may choose (non-canonically) an isomorphism
   \[ b_{\hf}^+: \Dcris\left(\cF^+ D(\hf)^*(-1-\hk) \right) \cong \cO_U. \]

   If $g$ is a non-critical-slope, cuspidal classical specialisation of $\hf$, with $g$ of weight $k+2$ for some $k \ne r$ then we have a \emph{canonical} isomorphism between the fibres of the above modules at $k$, given by the image modulo $\cF^-$ of the class in $\Fil^{k+1} \Dcris\left(V(g) \right)$ of the differential form $\omega_g$ associated to $g$. Here we use the comparison isomorphism between de Rham and \'etale cohomology crucially.

   So, for each such $g$, there must exist a non-zero constant $c_g \in L^\times$ such that $b_{\hf}^+$ specialises to $c_g \omega_g$.
%

   At $X = 0$, we have an isomorphism
   \[ \Dcris(\cF^+ D(f_\beta)^*) \cong \Dcris(V(f_\beta)^*_{\mathrm{quo}}), \]
   and we have a duality between $V(f_\beta)^*_{\mathrm{quo}}$ and $V^c(f_\beta)_{\mathrm{quo}}$; so we should regard $b_{f_\beta}^+$ as a basis of the space
   \[ \Dcris(V^c(f_\beta)_{\mathrm{quo}}) \cong H^1_{\mathrm{dR}, \mathrm{c}}(Y, \mathscr{V}_r)[T = f_\beta] = H^1(X, \omega^{-r}(-\mathrm{cusps}))[T = f_\beta].\]
   So there exists some scalar $c_{f_\beta} \in L^\times$ such that
   \[ b_{f_\beta}^+ = c_{f_\beta} \eta_{f_\beta}, \]
   where $\eta_f$ is as defined in \cref{sect:dRcoh}.

   \begin{remark}
    It is curious to note that the element $b_{\hf}^+$ ``generically'' interpolates the $\Fil^1$ vectors $\omega_g$ for specialisations $g$ of weight $\ne r$, but at the bad weight $k = r$, it interpolates the $\Fil^0$ vector $\eta_{f_\beta}$ instead.
   \end{remark}

\part{Euler systems and $p$-adic $L$-functions: background}
\addtocounter{section}{1}
In the next few sections, we recall the Euler systems and $p$-adic $L$-functions we shall use below. We present no new results here; but we will need to re-formulate some well-known results in minor ways, in order to be able to compare different constructions under our ``Eisenstein degeneration'' formalism in the final sections of this article.

\begin{assumption}\label{ass:classicalColeman}
 Throughout part B of this paper, we shall use $f_\beta$ to denote a classical \textbf{cuspidal} $p$-stabilised newform, of weight $r_1 + 2$ for some integer $r_1 \ge 0$, and \textbf{non-critical slope}; and we shall write $\hf$ for a Coleman family of eigenforms, over some open affinoid $V_1 \ni r$ in weight space, specialising to $f_\beta$ in weight $r_1$. We suppose, for convenience, that $f_\beta$ is a $p$-stabilisation of an eigenform of prime-to-$p$ level whose Hecke polynomial at $p$ has distinct roots (so $f_\beta$ is a ``noble eigenform'' in the sense of \cite{Hansen-Iwasawa}).
\end{assumption}

Similarly, $\hg, \hh$ will denote families over some affinoids $V_2$, $V_3$ passing through some given noble eigenforms $g_\beta$, $h_\beta$ of weights $r_1, r_2$. We shall allow ourselves to shrink the discs $V_i$ if needed, replacing them with arbitrarily small neighbourhoods of the $r_i$; this allows us to assume, for instance, that \emph{all} classical-weight specialisations of our families are noble cuspidal eigenforms.

(Of course, the primary goal of this paper is precisely to study the those families which do \emph{not} satisfy this condition, and this will be the goal of part C; but firstly we shall give a systematic account of the theory in the above setting, before explaining the modifications necessary for the critical-slope Eisenstein cases.)

\subsection{Setup}  We fix an embedding $\overline{\QQ} \into \overline{\QQ}_p$. As in part A, we shall fix a finite extension $L / \Qp$ with integers $\cO$. Let $\cH_\Gamma$ be the distribution algebra of $\Gamma \cong \Zp^\times$ (with $L$-coefficients), and $\hj: \Gamma \into \cH_\Gamma^\times$ the universal character, which we regard as a character of the Galois group by composition with the cyclotomic character.

\begin{notation}
 Abusively we shall write $H^1(\QQ, -)$ for $H^1(G_{\QQ, S}, -)$ where $S$ is a sufficiently large finite set of primes (i.e.~containing $p$ and all primes at which the relevant representation is ramified).
\end{notation}

For any $L$-linear $G_{\QQ}$-representation $V$, we write $V(-\hj)$ as a shorthand for $V \otimes_L \cH_{\Gamma}(-\hj)$. Thus, for any $\cO$-lattice $T \subset V$, we have a canonical isomorphism of $\cH_\Gamma$-modules
\[ H^1(\QQ, V(-\hj)) \cong \cH_{\Gamma} \otimes_{\Lambda_\Gamma} \varprojlim_n H^1(\QQ(\mu_{p^\infty}), T), \]
and similarly for $G_{\Qp}$-representations.
%

\subsection{Local machinery: Coleman--Perrin-Riou maps}

 Let $V$ be a crystalline $L$-linear representation of $G_{\Qp}$. Then there is a homomorphism of $\cH_\Gamma$-modules, the \emph{Perrin-Riou regulator},
 \[ \mathcal{L}^{\Gamma}_V: H^1(\Qp, V(-\hj)) \longrightarrow I^{-1} \cH_\Gamma \otimes_L \Dcris(\Qp, V) \]
 (which depends on a choice of $p$-power roots of unity $(\zeta_{p^n})_{n \ge 1}$ in $\overline{L}$, although we suppress this from the notation); here $I$ is a certain fractional ideal depending on the Hodge--Tate weights of $V$. The map $\mathcal{L}^{\Gamma}_V$ is characterised by a compatibility with the Bloch--Kato logarithm and dual-exponential maps for twists of $V$. See \cite{LVZ} for explicit formulae.

 If we choose a vector $\eta \in \Dcris(V)$, and apply the above constructions to $V^*$, then we obtain a \emph{Coleman map}
 \[ \operatorname{Col}_\eta: H^1(\Qp, V^*(-\hj)) \longrightarrow I^{-1} \cH(\Gamma), \qquad x \mapsto \langle \mathcal{L}^\Gamma_{V^*}(x), \eta \rangle. \]
 We shall most often use this when $V$ is unramified and non-trivial, in which case $I$ is the unit ideal.

 All three objects above -- the cohomology of $V(-\hj)$, the $\Dcris$ module, and the regulator map connecting them -- can also be defined more generally for crystalline $(\varphi, \Gamma)$-modules over the Robba ring, whether or not they are \'etale.

\section{The $\GL_2 / \QQ$ setting}
\label{es4}

 The Kato Euler system of \cite{kato04} can be attached to a Coleman family $\hf$ satisfying \cref{ass:classicalColeman}, as discussed e.g.~in \cite{Hansen-Iwasawa}. However, for our purposes it suffices to restrict to the case of Hida families, following the developments of \cite{ochiai03}. We suppose our family $\hf$ is new of some level $N$, and, as in \emph{op.cit.}, we assume that the Galois representation attached to $\hf$ is residually irreducible.

 \subsection{Periods}

  Let $f_k$ be (the newform associated to) the weight $k+2$ specialisation of $\hf$, for some $k \ge 0$. For each sign $\pm$, the eigenspace in Betti cohomology of $Y_1(N)$ (with $\QQ(f)$-coefficients) on which the Hecke operators act by the eigensystem of $f_k$ is one-dimensional, and we choose bases $\gamma_f^{\pm}$ of these spaces. These determine complex periods $\Omega^{\pm}_{f_k} \in \CC^\times$.

  Let $V_1$ be the open of the weight space over which the family is defined. Then we have an $\cO(V_1)$-module $V(\hf)^*$ interpolating the Betti (or \'etale) cohomology eigenspaces of all specialisations of $\hf$. Up to possibly shrinking $V_1$, we may assume that the eigenspaces $V(\hf)^{(c = \pm 1)}$ are free of rank 1 over $\cO(V_1)$, and choose bases $\gamma_{\hf}^{\pm}$. In general we cannot arrange that the weight $k$ specialisation of $\gamma_{\hf}^{\pm}$ is defined over $\QQ(f_k)$ for all $k$; so it is convenient to extend the definition of $\Omega^{\pm}_{f_k}$ accordingly, so these periods now lie in the space $(L \otimes_{\QQ(f_k)} \CC)^\times$.

  \begin{remark}
   For avoidance of confusion, we point out that the period we are calling $\Omega^{\pm}_{f_k}$ corresponds to $\tfrac{1}{\lambda^{\pm}(k)} \Omega^{\pm}_{f_k}$ in the notation of \cite[\S 3.2]{bertolinidarmon14}. Hence, although the periods we have considered are elements in $(L \otimes_{\QQ(f_k)} \CC)^\times$ (which will ease our subsequent discussions), it must be clear that one can easily recover more familiar objects from them.
  \end{remark}

 \subsection{The Kitagawa--Mazur $p$-adic $L$-function}

 Having chosen $\gamma_{\hf}^{\pm}$, we can define Kitagawa--Mazur-type $p$-adic $L$-functions \cite{kitagawa}
 \[ L_p(\hf) \in \cO(V_1 \times \cW), \]
 which interpolate the critical $L$-values of all classical, weight $\ge 2$ specialisations of $\hf$, with the periods determined by $\gamma^{\pm}_{\hf}$. More precisely, the value at $(k, j)$, with $0 \le j \le k$, is given by
 \[ L_p(\hf)(k, j) =  \frac{j!(1 - \tfrac{\beta_k}{p^{1+j}})(1 - \tfrac{p^j}{\alpha_k})}{(-2\pi i)^j \Omega_{f_k}^\pm} L(f_k, 1+j), \]
 where $f_k$ is the weight $k+2$ specialisation of $\hf$ as above, $(\alpha_k, \beta_k)$ are the roots of its Hecke polynomial (with $\alpha_k$ corresponding to the family $\hf$), and $\pm = (-1)^j$. (Here we assume $f_k$ is new of level $N$, which is automatic for $k > 0$; a slightly modified formula applies if $k = 0$ and $f_k$ is a newform of level $pN$ and Steinberg type at $p$.)

 \subsection{The adjoint $p$-adic $L$-function} We shall also need the following construction. For any classical specialisation $f$ of $\hf$, if we define the plus and minus periods $\Omega^{\pm}_{f}$ using $\QQ(f)$-rational basis vectors $\gamma_f^\pm$, then the ratio
 \[ L^{\mathrm{alg}}(\operatorname{Ad} f, 1) \coloneqq \frac{-2^{k-1} i \pi^2 \langle f, f \rangle}{\Omega^+_f \Omega^-_f} \]
 is in $\QQ(f)^\times$. If we choose bases $\gamma_{\hf}^{\pm}$ over the family as above, and use the periods for each $f_k$ determined by these, then a construction due to Hida \cite{hida-adjoint} gives a $p$-adic adjoint $L$-function $L_p(\operatorname{Ad} \hf) \in \cO(V_1)$ interpolating these ratios:
 \[ L_p(\operatorname{Ad} \hf)(k) = \left(1 - \tfrac{\beta_{k}}{\alpha_{k}}\right)\left(1 - \tfrac{\beta_{k}}{p\alpha_{k}}\right) L^{\mathrm{alg}}(\operatorname{Ad} f_k, 1)\]
 for all $k \in V_1 \cap \ZZ_{\ge 0}$ such that $f_k$ is a level $N$ newform. In particular, recall that if $\hf$ is $p$-distinguished, then the congruence ideal of $\hf$ is principal, and this ideal is generated by $L_p(\operatorname{Ad} \hf)$.

 \subsection{Kato's Euler system} Having chosen $\gamma_\hf^{\pm}$, we obtain a canonical \emph{Kato class}
 \[ \kappa(\hf) \in H^1(\QQ, V(\hf)^*(-\hj)), \]
 which is the ``$p$-direction'' of an Euler system. Kato's explicit reciprocity law \cite{kato04} establishes that the image of that class under the Perrin-Riou map recovers the $p$-adic $L$-function.

 More precisely, let $\cF^+ V(\hf)$ denote the rank 1 unramified subrepresentation of $V(\hf)$ over $\cO(V_1)$ (which exists since $\hf$ is ordinary); and let $\eta_{\hf} \in \Dcris(\cF^+ V(\hf))$ be the canonical vector constructed in \cite{KLZ17}, which is characterised by interpolating the classes $\eta_f$ of \cref{sect:dRcoh} for each classical specialisation $f$ of $\hf$. Then we have
 \[ \left\langle \mathcal{L}_{\cF^- V(\hf)^*}^\Gamma( \loc_p \kappa(\hf)), \eta_{\hf} \right\rangle = L_p(\hf). \]

 \begin{remark}
  This is a slight strengthening of a result of Ochiai; see \cite[Theorem 3.17]{ochiai03} for the original formulation. Ochiai's result is a little less precise, since he chooses an arbitrary basis $d$ of the module $\Dcris(\cF^+ V(\hf))$ (which is denoted $\mathcal{D}$ in \emph{op.cit.}); we have used the results on Eichler--Shimura in families proved in \cite{KLZ17} to choose a \emph{canonical} basis $\eta_{\hf}$, for which the correction terms $C_{p, \fp, d}$ in Ochiai's formulae are all 1.
 \end{remark}

\section{Double and triple products}
\subsection{The Rankin--Selberg setting}\label{es5}

Now let $\hf$ and $\hg$ be two Coleman families satisfying \cref{ass:classicalColeman}, living over discs $V_1$, $V_2$ in weight space.

\subsubsection{P-adic $L$-functions}

 There is an ``$\hf$-dominant'' $p$-adic $L$-function $L_p^{\hf}(\hf, \hg)$ over $V_1 \times V_2 \times \cW$, whose value at $(k, \ell, j)$ with $\ell+1 \le j \le k$ is $(\star) \cdot \frac{L(f_k, g_\ell, 1+j)}{\langle f_k, f_k \rangle}$, where $(\star)$ is the usual m\'elange of Euler factors, powers of $i$ and $\pi$ etc. Similarly, there is a ``$\hg$-dominant'' $p$-adic $L$-function $L_p^{\hg}(\hg, \hf)$, with an interpolating range at points with $k+1 \le j \le \ell$. If $\hf$ is ordinary (i.e.~a Hida family) then $L_p^{\hf}(\hf, \hg)$ is bounded.

\begin{remark}
 Note that our $p$-adic $L$-functions $L(f, g, s)$ here are \emph{imprimitive}, i.e.~their Dirichlet coefficients are given by the usual straightforward formula in terms of $q$-expansion coefficients, cf.~\cite[\S 2.7]{KLZ17}. This means they may differ by finitely many Euler factors from the \emph{primitive} $p$-adic $L$-function associated to the Galois representation $V_p(f) \otimes V_p(g)$ (although this issue only arises if the levels of $f$ and $g$ are not coprime).
\end{remark}

\subsubsection{Euler systems} The Beilinson--Flach Euler system of \cite{loefflerzerbes16} is attached to two modular forms, or more generally to two Coleman families $\hf$ and $\hg$. This generalizes the earlier construction of \cite{KLZ17}, where the variation was restricted to the case of ordinary families. Consider the rank 4 module
\begin{equation}
V(\hf, \hg)^*(-\hj) \quad\text{over}\quad \cO(V_1 \times V_2 \times \cW),
\end{equation}
characterized by the property that on specialising at any integers $(k,\ell,j)$, with $k \in V_1$ and $\ell \in V_2$, we recover $V(f_k)^* \otimes V(g_\ell)^*(-j)$.

Fix $d \in \ZZ_{>1}$ such that $(d,6S)=1$.  There exists a cohomology class of Beilinson--Flach elements \[ {}_d \kappa(\hf,\hg) \in H^1(\QQ, V(\hf, \hg)^*(-\hj)), \] which is the one in \cite[Theorem 5.4.2]{loefflerzerbes16}, where it would be denoted by ${}_d \mathcal{BF}_{1,1}^{[\hf, \hg]}$. (Note that we have slightly modified the notations just for being consistent with the other Euler systems, and we have written $d$ and not $c$ for the auxiliary parameter to avoid any misunderstanding with the index $c$ used to denote compactly-support cohomology.)

The dependence on $d$ is as follows: after tensoring with $\operatorname{Frac} \cO(V_1 \times V_2 \times \cW)$, the class
\[ \kappa(\hf, \hg) \coloneqq C_d^{-1} \otimes {}_d \kappa(\hf, \hg) \]
is independent of $d$, where
\begin{equation}
 \label{eq:defCd}
  C_d(\hf,\hg,\hj) \coloneqq d^2 - d^{(\hj-\hk_1-\hk_2)}\varepsilon_\hf^{-1}(d) \varepsilon_\hg^{-1}(d).
\end{equation}

\subsubsection{Reciprocity laws}\label{reclaw-bf}

With the notations used in \cite[\S 6]{loefflerzerbes16}, let $D(\hf)^*$ be the $(\varphi,\Gamma)$-module of $V(\hf)^*$, and consider the rank 1 submodule $\cF^+ D(\hf)^* \subset D(\hf)^*$ together with the corresponding quotient $\cF^- D(\hf)^*$. We consider the same filtration for $D(\hg)^*$. We write \[ \cF^{--} D(\hf,\hg)^* = \cF^- D(\hf)^* \hat \otimes \cF^- D(\hg)^*, \] and similarly for $\cF^{-+}$, $\cF^{+-}$, $\cF^{++}$. We also define $\cF^{- \circ} D(\hf,\hg)^* = \cF^- D(\hf)^* \hat \otimes D(\hg)^*$. Proceeding as in \cite[Theorem 7.1.2]{loefflerzerbes16}, the projection of ${}_d \kappa(\hf, \hg)$ to $\cF^{--}D(\hf,\hg)^*$ vanishes. Hence, the projection to $\cF^{-\circ}$ is in the image of the injection \[ H^1(\QQ, \cF^{-+} D(\hf, \hg)^*(-\hj)) \longrightarrow H^1(\QQ, \cF^{- \circ} D(\hf, \hg)^*(-\hj)). \]

Then the reciprocity law of Loeffler--Zerbes \cite[Theorem 7.1.5]{loefflerzerbes16} establishes that the image of that class under the Perrin-Riou map recovers the $p$-adic $L$-function.
\begin{equation}
 \label{eq:BFrecip}
 \operatorname{Col}_{\eta_{\hf} \otimes \omega_{\hg}}(\loc_p({}_d \kappa(\hf,\hg))) = C_d(\hf, \hg,\hj) \cdot L_p^{\hf}(\hf,\hg).
\end{equation}
where the Coleman map is the composition of the Perrin-Riou regulator followed by the pairing with the appropriate differentials.

\subsection{Diagonal cycles}\label{es6}

Let $(\hf,\hg,\hh)$ be a triple of Coleman families, living over discs $V_1$, $V_2$, $V_3$ in weight space, with all classical specialisations non-critical-slope cusp forms, as in the previous section.

We suppose the tame nebentype characters satisfy $\varepsilon_{\hf} \varepsilon_{\hg} \varepsilon_{\hh} = 1$. It follows that we may choose (non-uniquely) a character $\mathbf{t}: \Zp^\times \to \cO(V_1 \times V_2 \times V_3)$ satisfying $2\mathbf{t} = \hk_1 + \hk_2 + \hk_3$, where $\hk_i$ are the universal characters into $\cO(V_i)$. This imposes an additional condition on our specialisations: we shall say a point $P$ of $V_1 \times V_2 \times V_3$ is an ``integer point'' if $(\hk_1, \hk_2, \hk_3)$ specialise at $P$ to integers $(k, \ell, m) \ge -1$, and, in addition, $\mathbf{t}$ specialises to $\tfrac{k + \ell + m}{2}$ (rather than its product with the quadratic character mod $p$), which amounts to requiring that $k + \ell + m$ lie in a particular congruence class modulo $2(p-1)$.

\subsubsection{P-adic $L$-functions}

Following a construction of Andreatta and Iovita \cite{andreattaiovita21}, there is a $\hf$-dominant square root $p$-adic $L$-function $\Lp^{\hf}(\hf,\hg,\hh)$ over $V_1 \times V_2 \times V_3$.  The square of its value at an integer point $(k, \ell, m)$ with $k > \ell+m$ is $(\star) \cdot L(f_k, g_\ell, h_m, \tfrac{k+\ell+m}{2}+2)$, where $(\star)$ is the usual m\'elange of Euler factors, complex periods etc. (Note that with our conventions $f_k$ has weight $k+2$, etc, so $\tfrac{k+\ell+m}{2}+2$ is the centre of the functional equation.)

\begin{remark}
In fact we shall only need this construction when the ``dominant'' family is an ordinary family. In this case, the construction is actually considerably simpler, and can be carried out via the same techniques as in the ordinary case, without need to resort to the developments of Andreatta--Iovita.
\end{remark}

\subsubsection{Euler systems}

Consider the rank 8 module
\begin{equation}
V(\hf, \hg, \hh)^* \coloneqq V(\hf)^* \otimes V(\hg)^* \otimes V(\hh)^*(-1-\tfrac{\hk_1 + \hk_2 + \hk_3}{2}) \quad\text{over}\quad \cO(V_1 \times V_2 \times V_3),
\end{equation}
which is Tate self-dual. By the works of Darmon--Rotger \cite{DR3} and Bertolini--Seveso--Venerucci \cite{BSV1}, there is a \textbf{diagonal-cycle class} $\kappa(\hf,\hg,\hh)$ attached to the triple $(\hf,\hg,\hh)$ (and the choice of square-root character $\mathbf{t}$, which we suppress); this is a class
\[ \kappa(\hf,\hg,\hh) \in H^1(\QQ, V(\hf,\hg,\hh)^*), \]
introduced for instance in \cite[\S 8.1]{BSV1}. It is characterized by the property that on specialising at any integer point $(k,\ell,m)$, with $k \in V_1$, $\ell \in V_2$ and $m \in V_3$ satisfying the ``balanced'' conditions $\{ k \le m + \ell, \ell \le m + k, m \le k + \ell\}$, we recover the Abel--Jacobi image of the diagonal cycle for $(f_k, g_\ell, h_m)$ defined in \cite{darmonrotger14}.

\begin{remark}\label{rmk:notperfect}
 Note that there is an omission in section 4.2 of \cite{BSV1}, where the machinery for interpolating diagonal-cycle classes is developed: it is implicitly assumed that the Ohta pairing $V(\hf) \times V^c(\hf)^* \to \cO(V_1)$ (formula (82) in \emph{op.cit.}) is is perfect. This cannot be be true in general, since at a critical-slope Eisenstein point it specialises to the second of the two pairings in \eqref{eq:notperfect} above, and we have seen that this has 1-dimensional left and right kernels. However, the pairing does become perfect after specialising to a neighbourhood of a noble eigenform, and hence the construction is valid when $\hf, \hg, \hh$ satisfy the conditions of \cref{ass:classicalColeman}; see the erratum \cite{BSV1-erratum} for the necessary modifications.

 We shall see in Part C of this paper that the construction gives something slightly weaker, i.e. a cohomology class taking values in $\tfrac{1}{X} V^c(\hf, \hg, \hh)^* \supset V(\hf, \hg, \hh)^*$, when one of the families is critical-slope.
\end{remark}

\subsubsection{Reciprocity laws} With the notations of \cref{reclaw-bf}, we may consider the $(\varphi,\Gamma)$-module $D(\hf,\hg,\hh)^*$, as well as its different filtrations. In particular, adapting \cite[Corollary 8.2]{BSV1} to our setting yields that the diagonal cycle class lies in the rank four submodule $(\cF^{++\circ} + \cF^{+\circ+} + \cF^{\circ ++})D(\hf,\hg,\hh)^*$. In particular, the projection to $\mathcal F^{- \circ \circ}$ is in the image of the injection \[ H^1(\QQ, \cF^{-++} D(\hf, \hg, \hh)^*) \longrightarrow H^1(\QQ, \cF^{- \circ \circ} D(\hf, \hg,\hh)^*). \]

Then the reciprocity law of Bertolini--Seveso--Venerucci establishes a connection between the diagonal-cycle class and the $p$-adic $L$-function.

\begin{theorem}\label{thm:rec-law-cycles}
The image of that class under the corresponding Coleman map recovers the $p$-adic $L$-function.\[ \operatorname{Col}_{\eta_{\hf} \otimes \omega_{\hg} \otimes \omega_{\hh}}(\loc_p(\kappa(\hf,\hg,\hh))) = \Lp^{\hf}(\hf,\hg,\hh). \]
\end{theorem}

\begin{proof}
This follows from \cite[Theorem A]{BSV1}. Although the result in \emph{loc.\,cit.} is stated just for Hida families, the argument remains valid in the present generality, using the $p$-adic $L$-function of Andreatta--Iovita \cite{andreattaiovita21} and the general theory of $(\varphi,\Gamma)$-modules.
\end{proof}

\section{$\GL_1$ over an imaginary quadratic field}
\label{es2}

\subsection{Setting} We fix an imaginary quadratic field $K$ where the prime $p$ splits. For simplicity, we shall always take $K$ to have class number one. We also fix embeddings $K \subset \overline{\QQ} \hookrightarrow \Qpb$ (determining a prime $\fp$ of $K$ above $p$).

We write $\sigma$ and $\bar\sigma$ for the chosen embedding $K \into \QQ_p$ and its complex conjugate. If $a - b = 0 \bmod w_K \coloneqq \#\cO_K^\times$, then the character $\sigma^a \bar{\sigma}^b$ of $K^\times$ is trivial on $\cO_K^\times$, and thus determines an algebraic Gr\"ossencharacter of $K$ of conductor 1 (mapping a prime $\mathfrak{q}$ to $\lambda^a \bar{\lambda}^b$ for $\lambda$ any generator of $\mathfrak{q}$). We abuse notation slightly by writing $\sigma^a \bar{\sigma}^b$ for this Gr\"ossencharacter. Its infinity-type (with the conventions of \cite{BDP13} and \cite{JLZ} \S 2.1) is $(a, b)$.

We let $\Sigma_K$ be the set of algebraic Gr\"ossencharacters of $K$. Define $\Sigma_K^{\crit} = \Sigma_K^{(1)} \cup \Sigma_K^{(2)} \subset \Sigma_K$ to be the disjoint union of the sets
\begin{align*}
 \Sigma_K^{(1)} &= \{ \xi \in \Sigma_K \text{ of infinity type } (a,b), \text{ with } a \leq 0, b \geq 1 \}, \\
 \Sigma_K^{(2)} &= \{ \xi \in \Sigma_K \text{ of infinity type } (a,b), \text{ with } a \geq 1, b \leq 0 \}.
\end{align*}
Thus $\xi \in \Sigma_K^{\crit}$ if and only if $s=0$ is a critical point for $L(\xi^{-1},s)$.

For $\Psi$ any algebraic Gr\"ossencharacter of $K$, we may define a 1-dimensional Galois representation $V_p(\Psi)$, on which $G_K$ acts via the composite of $\Psi$ with the Artin map (normalised to send geometric Frobenii to uniformizers, as in \cite[\S 2.3.2]{JLZ}). In particular $V_p(\sigma \bar\sigma) \cong \Qp(-1)$, and $L(V_p(\Psi), s) = L(\Psi, s)$.

\subsection{Character spaces}
Let $\Gamma_K$ denote the group $\Gal(K[p^\infty] / K)$, where $K[p^\infty]$ is the ray class field mod $p^\infty$; and let $\mathcal W_K$ be the corresponding character space, so $\cO(\cW_K) = \cH(\Gamma_K)$. We let $\hj_K$ be the universal character $\Gal(K^{\mathrm{ab}} / K) \twoheadrightarrow \Gamma_K \into \cH(\Gamma_K)^\times$.

We identify an algebraic Gr\"ossencharacter $\xi$ of $K$ unramified outside $p$ with the unique point of $\cW_K$ at which $\hj_K$ specialises to $V_p(\xi^{-1})$ (note signs).

\begin{remark}
 This inverse ensures that, if we identify $\Gamma_K$ with $\cO_{K,p}^\times / \cO_K^\times$ via the restriction of the Artin map to $\cO_{K, p}^\times \subset \mathbf{A}_{K, \mathrm{f}}^\times$, the character $x \mapsto \operatorname{Nm}_{K/\QQ}(x)$ corresponds to the cyclotomic character.
\end{remark}

\subsection{Katz's $p$-adic $L$-function}

Let $\Psi$ be a Gr\"ossencharacter of finite order and conductor coprime to $p$, with values in $L$. By the work of Katz \cite{katz76} (see also \cite[Theorem 3.1]{BDP12}), there exists a $p$-adic analytic function
\[ L_{\fp}^{\mathrm{Katz}}(\Psi): \cW_K \longrightarrow L \otimes_{\Qp} \widehat{\QQ}_p^{\mathrm{nr}}, \]
uniquely determined by the interpolation property that if $\xi \in \Sigma^{(2)}_K$ is a character of conductor 1, hence necessarily of the form $\sigma^a \bar{\sigma}^b$ for some $a \ge 1, b \le 0$, then we have
\[ \frac{L_{\fp}^{\mathrm{Katz}}(\Psi)(\xi)}{\Omega_p^{a-b}} = \mathfrak a(\xi) \times \mathfrak e(\xi) \times \mathfrak f(\xi) \times \frac{L(\Psi\xi^{-1},0)}{\Omega^{a-b}}, \]
with both sides lying in $\overline{\QQ}$, where
\begin{enumerate}
\item $\mathfrak a(\xi) = (a-1)! \pi^{-b}$,
\item $\mathfrak e(\xi) = (1- p^{-1} \Psi^{-1}\xi(\fp))(1-\xi^{-1}\Psi(\bar{\fp}))$,
\item $\mathfrak f(\xi) = D_K^{b/2} 2^{-b}$,
\item $\Omega_p \in  (\widehat{\QQ}_p^{\mathrm{nr}})^\times$ is a $p$-adic period attached to $K$,
\item $\Omega \in \CC^{\times}$ is the complex period associated with $K$,
\item $L(\Psi \xi^{-1},s)$ is Hecke's $L$-function associated with $\Psi \xi^{-1}$.
\end{enumerate}

(More generally, one can state an interpolation property at all algebraic Gr\"ossencharacters $\xi \in \Sigma^{(2)}_K$ of $p$-power conductor -- not necessary of conductor 1 -- but we shall not need this more general formula here.)

\subsection{Elliptic units}

 Again, let $\Psi$ be a Gr\"ossencharacter of finite order and conductor coprime to $p$. Let $(-)^\sim$ denote the reflexive hull of a $\cH_\Gamma$-module. The Euler system of elliptic units can be thought as an element
 \[ \kappa(\Psi,K) \in H^1(K, V_p(\Psi)^*(1-\hj_K))^{\sim}, \]
 constructed by Coates and Wiles in their seminal paper \cite{coateswiles77}; see \cite[\S 15]{kato04} or \cite[\S 1.2]{BCDDPR} for more recent accounts\footnote{Note that the reflexive closure seems to have been overlooked in the latter reference; it is not needed if $\Psi$ is ramified at some prime away from $p$, but cannot be got rid of when $\Psi$ has conductor 1. In Kato's account this corresponds to passing from the ``smoothed'' class ${}_\mathfrak{a} z_{p^\infty \mathfrak{f}}$ to its analogue without the modification ${}_\mathfrak{a}()$.} By construction, the specialisation of $\kappa(\Psi,K)$ at a finite-order character $\xi$ of $\Gamma_K$ is the image under the Kummer map of a linear combination of global units in an abelian extension of $K$.

 Localising at $\fp$, and using the two-variable version of Perrin-Riou's regulator defined in \cite{loefflerzerbes14}, we have a Coleman map
 \[ \operatorname{Col}_{\fp, \Psi}: H^1(K_{\fp}, V_p(\Psi)^*(1-\hj_K)) \to \cH(\Gamma_K),\]
 which extends automatically to the reflexive hull.  
 The explicit reciprocity law of Coates--Wiles links the system of elliptic units with Katz's two variable $p$-adic $L$-function:
 \[ \operatorname{Col}_{\fp, \Psi}\Big(\loc_{\fp} \kappa(\Psi,K)\Big) = L_{\fp}^\mathrm{Katz}(\Psi). \]

\section{Heegner classes}\label{es3}

We follow for this section the exposition of \cite{JLZ}, which generalizes Castella's earlier results \cite{castella-variation}, and keep the notations of the previous sections. We suppose all primes dividing $N$ are split in $K$, and choose an ideal $\mathfrak{N}$ with $\cO_K / \mathfrak{N} \cong \ZZ / N$. Let $\hf$ be a Coleman family of tame level $N$ defined over an affinoid disc $V_1$.

In \cite{JLZ}, we worked over an auxiliary rigid space $\widetilde{V}_1$ (essentially a piece of the eigenvariety for $\operatorname{GU}(1, 1)$) parametrising conjugate-self-dual twists of the base-change of $\hf$ to $K$. To simplify the exposition, and since this is harmless towards our objectives of explaining the degeneration phenomena going on, we shall avoid appealing to this construction by make the following two simplifying assumptions:
\begin{enumerate}[(a)]
 \item $K$ has class number one, as above, and moreover $K \ne \QQ(\sqrt{-1}), \QQ(\sqrt{-3})$;
 \item $\hf$ has trivial nebentype.
\end{enumerate}
These simplifying hypotheses will allow us to split the parameter space up into two copies of $\cW$, one for the ``weight'' variable and one for the ``anticyclotomic'' one.

More precisely, assumption (b) allows us to choose a character $\mathbf{m}: \Zp^\times \to \cO(V_1)^\times$ which is a square root of the canonical character $\hk$, so $\hf$ has weight-character $\hk + 2 = 2\mathbf{m} + 2$. Slightly abusively, we write ``$\tfrac{\hk}{2}$'' for this character. Hence the representation $W = V(\hf)^*(-\tfrac{\hk}{2})$ of $G_{\QQ}$ satisfies $W \cong W^*(1)$.

\begin{remark}
 Note that if $k$ is an integer in $V_1$, then $k$ is necessarily even. However, the specialisation of ``$\tfrac{\hk}{2}$'' at $k$ might not be $x \mapsto x^{k/2}$; in fact this holds only if $k$ lies in a certain congruence class mod $2(p-1)$, lifting the congruence class mod $(p-1)$ determined by the component of $\cW$ containing $V_1$.
\end{remark}

Meanwhile, assumption (a) implies that the Gr\"ossencharacter $\chi_{\ac} = \sigma/\bar{\sigma}$ gives an isomorphism $\Gamma^{\ac} \cong \Zp^\times$, where $\Gamma^{\mathrm{ac}} = \cO_{K, p}^\times / \Zp^\times$ is the Galois group of the maximal anticyclotomic extension unramified outside $p$. Composing this with the universal character $\hj: \Zp^\times \to \cO(\cW)^\times$, we obtain a universal anti-cyclotomic character $\hj^{\ac}: G_K \to \cO(\cW)^\times$, whose specialisation at $j \in \cW$ is the character mapping arithmetic Frobenius at a prime $\lambda \nmid p$ of $K$ to $\left(\sigma(\lambda) / \bar{\sigma}(\lambda)\right)^j$.

Consider the module over $\cO(V_1 \times \cW)$ defined by
\[ V^{\ac}(\hf)^* := V(\hf)^*(-\tfrac{\hk}{2}) \mathop{\hat\otimes}_L \cO(\cW)(-\hj^{\ac}).\]
The Galois representation $V^{\ac}(\hf)^*$ is characterized by the property that for any integers $(k,j)$, with $k \in V_1$ (and $k$ in the appropriate congruence class modulo $2(p-1)$), we recover the Galois representation
\[  V^{\ac}(\hf)^*(k,j) = V_p(f_k \times \chi^{-1})^*,\qquad \chi = \sigma^{(k/2 + j)} \bar{\sigma}^{(k/2-j)}, \]
where $f_k$ is the weight $k+2$ specialisation of $\hf$.

\begin{remark}
 Note that $(f_k, \chi)$ is a \emph{Heegner pair} in the sense of \cite{JLZ}. This corresponds to the fact that the character $\chi' = \chi \cdot \operatorname{Nm}$ is \emph{central critical} for $f_k$, in the sense of \cite{BDP13}, so that $L(f, \chi^{-1}, 1) = L(f, (\chi')^{-1}, 0)$ is a central $L$-value; cf.~Remark 2.2.2 of \cite{JLZ} for the shift by 1. If $-\tfrac{k}{2} \le j \le \tfrac{k}{2}$, then $\chi'$ is a point of the region $\Sigma^{(1)}_{\mathrm{cc}}(f_k)$ in Figure 1 of \cite{BDP13}. If $j \ge \tfrac{k}{2}+1$, it is in $\Sigma^{(2)}_{\mathrm{cc}}(f_k)$. (Similarly, it is in $\Sigma^{(2')}_{\mathrm{cc}}(f)$ if $j \le -1-\tfrac{k}{2}$, but we shall not consider this region here.)
\end{remark}

In this scenario, we have:
\begin{itemize}
 \item a \textbf{Heegner class}
 \[ \kappa(\hf,K) \in H^1(K, V^{\ac}(\hf)^*)\]
 constructed in \cite{JLZ} (see Theorem A), whose specialisations at $(k, j)$ with $-\tfrac{k}{2} \le j \le \tfrac{k}{2}$ are the Abel--Jacobi images of Heegner cycles.
 \item an \textbf{anticyclotomic $p$-adic $L$-function}
 \[ L_{\fp}^\mathrm{BDP}(\hf)\in \cO(V_1 \times \cW), \]
 such that the square of its value at a point $(k,j)$ with $j \geq \tfrac{k}{2} + 1$ agrees with $(*) \cdot L(f_k/K \times \chi^{-1}, 1)$, where $(*)$ the usual combination of Euler factors.
 \item an \textbf{explicit reciprocity law}, \cite[Theorem B]{JLZ}. To state this, we note that $\loc_{\fp} \kappa(\hf, K)$ factors through the anticyclotomic Iwasawa cohomology of a rank 1 submodule $\cF^+_p D(\hf)^* \subset D(\hf)^*$, where $D(\hf)^*$ is the $(\varphi, \Gamma)$-module of $V(\hf)^*$ at $p$ (cf.~Theorem 6.3.4 of \emph{op.cit.}). Letting $\cF^-_p D(\hf)$ denote the quotient of $D(\hf)$ dual to this, there is a canonical basis vector $\omega_{\hf}$ of $\Dcris(\cF^-_p D(\hf))$, interpolating the classes $\omega_f$ for classical specialisations $f$ of $\hf$ (cf.~\cref{sect:dRcoh}), which we may use this to define a Coleman map $\operatorname{Col}_{\fp, \omega_{\hf}}$. The explicit reciprocity law then states that
 \[ \operatorname{Col}_{\fp, \omega_{\hf}}(\loc_{\fp}(\kappa(\hf,K))) = (-1)^{(\hk/2+\hj)} L_p^\mathrm{BDP}(\hf). \]
\end{itemize}

\part{Critical-slope Eisenstein specialisations}

 \section{Deformation of Beilinson--Flach elements}
\label{sect:BFdef}

  \subsection{Setup} In this section we consider the Beilinson--Flach Euler system of \cite{loefflerzerbes16} attached to two modular forms, or more generally to two Coleman families. This generalizes the earlier construction of \cite{KLZ17}, where the variation was restricted to the case of ordinary families.

  We review some of the notation already introduced in previous parts to make the section more self-contained. Let $f = E_{r+2}(\psi, \tau)$ stand for the Eisenstein series of weight $r+2$ and characters $(\psi,\tau)$, and $f_\beta$ its critical-slope $p$-stabilisation. Under the non-criticality conditions discussed in \cref{equivalences}, there is a unique Coleman family $\hf$ passing through $f_\beta$, over some affinoid disc $V_1 \ni r$. We may suppose that for all integers $k \in V_1 \cap \ZZ_{\ge 0}$ with $k \ne r$, the specialisation $f_k$ is a non-critical-slope cusp form.

Meanwhile, let $\hg$ be a second Coleman family over some disc $V_2$. We suppose for simplicity that $\hg$ is ordinary. Let $V(\hf, \hg)^*$ be the module defined in \cref{es5}. Recall that the Galois representation $V(\hf, \hg)^*$ is characterized by the property that for any integers $(k,\ell,j)$, with $k \in V_1$ and $\ell \in V_2$, we recover \[ V(\hf, \hg)^*(k,\ell,j) = V(f_k)^* \otimes V(g_\ell)^*(-j), \] the $(-j)$-th Tate twist of the tensor product of the dual Galois representations attached to $f_m$ and $g_\ell$, as defined in \cref{rep-hida} (including the case $k = r$, when $f_k = f_\beta$). We define similarly a space $V^c(\hf, \hg)^*$ using $V^c(\hf)^*$ in place of $V(\hf)^*$. Note that these become isomorphic after inverting $X$.

  \subsection{Selmer vanishing}

   The representation $V(\hg)^*$ has a canonical rank-one $G_{\Qp}$-subrepresentation $\cF^+ V(\hg)^*$, with unramified quotient $\cF^- V(\hg)^*$; and we have the following:

   \begin{proposition}
    \label{prop:Selvanish}
    For any integer $n$ and Dirichlet character $\chi$, the ``Greenberg Selmer group''
    \[ H^1_{\mathrm{Gr}}(\QQ, V(\hg)^*(\chi)(n)\otimes \cH_{\Gamma}(-\hj)) \coloneqq \operatorname{ker}\Bigg(H^1(\QQ, V(\hg)^*(\chi)(n) \otimes \cH_{\Gamma}(-\hj) ) \to H^1(\Qp, \cF^-V(\hg)^*(\chi)(n) \otimes \cH_{\Gamma}(-\hj) )\Bigg) \]
    vanishes.
   \end{proposition}

   \begin{proof}
    We may take $n = 0$ and $\chi = 1$ without loss of generality. The result now follows from Kato's theorems \cite{kato04}, which show that for each classical specialisation $g_\ell$ of $\hg$, the module $H^1(\QQ, V(g_\ell)^* \otimes \cH_{\Gamma}(-\hj) )$ is free of rank 1 over $\cH_{\Gamma}$ and contains a canonical element (Kato's Euler system for $g_\ell$) whose localisation at $p$ is mapped to the $p$-adic $L$-function of $g_\ell$ under the Perrin-Riou regulator for $\cF^- V(g_\ell)^*$. Since the $p$-adic $L$-function is not a zero-divisor, we conclude that the space $H^1_{\mathrm{Gr}}(\QQ, V(g_\ell)^*\otimes \cH_{\Gamma}(-\hj))$ vanishes for each such $g_\ell$.

    So any element of $H^1_{\mathrm{Gr}}(\QQ, V(\hg)^*\otimes \cH_{\Gamma}(-\hj))$ must specialise to 0 at a Zariski-dense set of points of $V_2$. On the other hand, this module is contained in the full $H^1$, which is $\cO(V_2)$-torsion-free, by the exact sequence associated to multiplication by an element of $\cO(V_2)$. Hence the Greenberg Selmer group vanishes.
   \end{proof}

   \begin{remark}
    It is slightly irritating that our analysis of the specialisation of Beilinson--Flach elements relies on these Selmer-group bounds, and thus on the existence of Kato's Euler system. This would be an obstacle if we wanted to use the techniques of ``critical-slope Eisenstein specialisations'' to define \emph{new} Euler systems (rather than obtaining relations between existing Euler systems).
   \end{remark}

  \subsection{Families over punctured discs}

  \begin{proposition}
   The cohomology $H^1(\QQ, V(\hf, \hg)^*)$ is a finitely-generated module over $\cO(V_1 \times V_2 \times \cW)$, and this module is $X$-torsion-free, where $X \in \cO(V_1)$ is a uniformizer at $r$.
  \end{proposition}

  \begin{proof}
   This follows via the exact sequence of cohomology from the vanishing of $H^0(\QQ, V(\hf, \hg)^* / X) = H^0(\QQ, V(f_\beta)^* \otimes V(\hg)^*\otimes \cH_{\Gamma}(-\hj))$.
  \end{proof}

  \begin{theorem}\label{thm:bf-class}
   Fix $d \in \ZZ_{>1}$ such that $(d,6S)=1$.  There exists a cohomology class
   \[ {}_d \kappa(\hf,\hg) \in H^1(\QQ, \tfrac{1}{X} V^c(\hf, \hg)^*), \]
   with the following interpolation property:
   \begin{itemize}
    \item If $(k, \ell)$ are integers $\ge 0$ with $k \ne r$, then we have
    \[ {}_d \kappa(\hf,\hg)(k, \ell) = {}_d \mathcal{BF}^{[f_k, g_\ell]} \in H^1(\QQ, V(f_k)^* \otimes V(g_\ell)^* \otimes \cH_\Gamma(-\hj)),\]
    where the element ${}_d \mathcal{BF}^{[f_k, g_\ell]} = {}_d \mathcal{BF}^{[f_k, g_\ell]}_{1, 1}$ is as defined in Theorem 3.5.9 of \cite{loefflerzerbes16}.
   \end{itemize}
  \end{theorem}

  \begin{proof}
   Compare \cite[Theorem 5.4.2]{loefflerzerbes16}, which is an analogous result when all integer-weight specialisations of $\hg$ are classical. In the present situation, we must be a little more circumspect, since Proposition 5.2.5 of \emph{op.cit.} does not apply for $k = r$; so the map denoted $\operatorname{pr}_{\hg}^{[j]}$ there is not defined for $j = k + 1$. However, inverting $X$ gets rid of this problem.
  \end{proof}

  \subsection{Local properties at $p$}

   \begin{proposition}
    The image of $\operatorname{loc}_p\left({}_d \kappa(\hf,\hg)\right)$ in $H^1\left(\Qp, \tfrac{1}{X}\cF^{--} D^c(\hf, \hg)^*\right)$ is zero.
   \end{proposition}

   \begin{proof}
    This follows from the fact that the Iwasawa cohomology is torsion-free, and the specialisations away from $X = 0$ have the required vanishing property.
   \end{proof}

  \subsection{Leading terms at $X = 0$}

In the following proposition, we consider the quotient $\frac{\tfrac{1}{X} V^c(\hf, \hg)^*}{V(\hf, \hg)^*}$, which makes sense by the discussion of Section \ref{inclusions} leading to the chain of inclusions \eqref{chain}.

   \begin{proposition}
    The image of ${}_d \kappa(\hf,\hg)$ in the cohomology of the quotient
    \[ \frac{\tfrac{1}{X} V^c(\hf, \hg)^*}{V(\hf, \hg)^*} \cong \Qp(\tau^{-1})(1+r) \otimes V(\hg)^* \otimes \cH_{\Gamma}(-\hj) \]
    is zero.
   \end{proposition}

   \begin{proof}
   Firstly, observe that the isomorphism of the statement follows from Section \ref{inclusions}. We consider the projection of this class to the local cohomology at $p$ of the quotient $\cF^- V(\hg)^*$. Since the $(\varphi, \Gamma)$-module of $V^c(f_\beta)^*_{\mathrm{quo}}$ injects into $\cF^-D^c(f_\beta)$, this projection is 0, by \cref{prop:Selvanish}. Hence the global class lands in the Greenberg Selmer group, which is zero, as we have seen.
   \end{proof}

   \begin{corollary}
    The class ${}_d \kappa(\hf,\hg)$ lifts (uniquely) to $H^1(\QQ, V(\hf, \hg)^*)$, and thus has a well-defined image
    \[ {}_d \hat{\kappa}(f_\beta, \hg) \in H^1(\QQ, \Qp(\psi^{-1}) \otimes V(\hg)^* \otimes \cH_{\Gamma}(-\hj)).\]
   \end{corollary}

For the following result, recall the logarithmic distribution, as introduced for instance in \cite[\S1]{bellaichedasgupta15}: for a continuous character $\sigma : \ZZ_p^{\times} \rightarrow \CC_p$, the function $\frac{d^k \sigma}{dz^k} \cdot \frac{z^k}{\sigma(z)}$ is a constant, and $\log^{[k]} \in \CC_p$ is defined to be this constant.

   \begin{proposition}\label{div-log}
    The class ${}_d \hat{\kappa}(f_\beta, \hg)$ is divisible by the (cyclotomic) logarithm distribution $\log^{[r+1]} \in \cH_{\Gamma}$.
   \end{proposition}

   \begin{proof}
    The specialisation of ${}_d \kappa(\hf,\hg)$ at a locally-algebraic character of $\Gamma$ of degree $j \in \{0, \dots, r\}$ factors through the image of $\sD_{U - j} \otimes \operatorname{TSym}^j$ in $\sD_{U - (r+1)} \otimes \operatorname{TSym}^{(r+1)}$, and the maps $\Pr_{\hf}^{[j]}$ and $\Pr_{\hf}^{[r+1]}$ agree on this image up to a non-zero scalar. Since the $\Pr_{\hf}^{[j]}$ for $0 \le j \le r$ do not have poles at $X = 0$, it follows that the specialisations of ${}_d \kappa(\hf,\hg)$ at triples $(r, \ell, \chi)$, for $\ell \ge r$ and $\chi$ locally-algebraic of degree $\in \{0, \dots, r\}$, interpolate the projections of the classical Beilinson--Flach classes to the $(E_{r+2}^{\crit}, g_\ell)$-eigenspaces in classical cohomology. Since the Beilinson--Flach classes arise as suitable projections of classes in the cohomology of $X_1(N) \times Y_1(N)$, these projections are always 0. By Zariski-density, the class specialises to 0 everywhere in $\{r \} \times V_2 \times \{\chi\}$. Since this holds for all $\chi$ of degree up to $r$, and these are exactly the zeroes of $\log^{[r+1]}$, the result follows.
   \end{proof}

   Since the Iwasawa cohomology is torsion-free, there is a unique class
   \[ {}_d \kappa(f_\beta, \hg) \in H^1(\QQ, V(\hg)^*(\psi^{-1}) \otimes \cH_{\Gamma}(-\hj))\]
   such that
   \[ {}_d \hat{\kappa}(f_\beta,\hg) = \log^{[r+1]} \cdot {}_d \kappa(f_\beta,\hg). \]
   Moreover, since $\hg$ is ordinary, the class ${}_d \kappa(f_\beta,\hg)$ has bounded growth and hence lies in $H^1(\QQ, V(\hg)^*(\psi^{-1}) \otimes \Lambda_{\Gamma}(-\hj))$, where $\Lambda_{\Gamma}$ is the cyclotomic Iwasawa algebra. (More generally, we could carry this out with a non-ordinary family $\hg$, and we would obtain a class with growth of order equal to the slope of $\hg$.)

 \subsection{The $p$-adic $L$-function}

We consider the `$\hg$-dominant' $p$-adic $L$-function $L_p^{\hg}(\hf, \hg)$ over $V_1 \times V_2 \times \cW$. The interpolation property applies also at $k = r$ without any special complications; and here the complex $L$-function factors as
  \begin{equation}
   \label{eq:Artin}
   L(E_{r+ 2}(\psi, \tau), g_\ell, 1 + j) = L(g_\ell, \psi, 1+j) \cdot L(g_\ell, \tau, j - r).
  \end{equation}
  Note that both factors on the right-hand side are critical values. Thus the restriction of $L_p^{\hg}(\hf, \hg)$ to the $k = r$ fibre is uniquely determined by its interpolation property at crystalline points (we don't need finite-order twists), and we have an ``Artin formalism'' factorisation
  \[ L_p^{\hg}(E_{r + 2}(\psi, \tau), \hg)(\hj) = \frac{L_p(\hg \times \psi, \hj)\cdot  L_p(\hg \times \tau, \hj - 1 - r)}{L_p(\operatorname{Ad} \hg)}, \]
  where the denominator arises from the choice of periods.

 \subsection{Perrin-Riou maps}

  We want to relate $L_p^{\hg}(\hf, \hg)$ to the image of $\operatorname{loc}_p({}_d \kappa(\hf,\hg))$ under the projection to $\cF^- V(\hg)^*$. As discussed in \cref{reclaw-bf}, this factors through the natural map
  \[ H^1_{\mathrm{Iw}}(\QQ_{p, \infty}, \cF^{+-} D(\hf, \hg)^*) \to H^1_{\mathrm{Iw}}(\QQ_{p, \infty}, \cF^{\circ -} D(\hf, \hg)^*),\]
  which is injective (since $H^0_{\mathrm{Iw}}(\cF^{--})$ will be zero). Perrin-Riou's regulator gives us a map
  \[
   \operatorname{Col}_{\mathbf{b}_{\hf}^+ \otimes \eta_{\hg}} = \left\langle \mathcal{L}^{\mathrm{PR}}_{\cF^{+-}}(-), \mathbf{b}_{\hf}^+ \otimes \eta_{\hg}\right\rangle :
   H^1_{\mathrm{Iw}}(\QQ_{p, \infty}, \cF^{+-} D(\hf, \hg)^*) \to \cO(V_1 \times V_2 \times \cW)
  \]
  which interpolates the Perrin-Riou regulators for $f_k \times g_\ell$ for varying $(k, \ell)$.

  Let $\varphi^{-1}$ stand for the left inverse of the Frobenius, denoted as $\psi$ in \cite[\S8.2]{KLZ17}. More precisely, proceeding as in \emph{loc.cit.} and using Fontaine isomorphism, for $z \in \left(\cF^{+-} D(\hf, \hg)^*\right)^{\varphi^{-1} = 1}$, this map sends $z$ to
  \[ \langle\iota((1 - \varphi) z), \mathbf{b}_{\hf}^+ \otimes \eta_{\hg}\rangle, \]
  where $\iota$ is the inclusion
  \[
   \left(\cF^{+-} D(\hf, \hg)^*\right)^{\varphi^{-1} = 0} \into \left(\cF^{+-} D(\hf, \hg)^*[1/t]\right)^{\varphi^{-1} = 0}
   = \Dcris\left(\cF^{+}D(\hf)^*(-1-\hk)\right)\otimes \Dcris\left(\cF^{-}D(\hg)^*\right) \otimes \cH_{\Gamma}.
  \]
  For the following discussion, recall the constants introduced in \ref{subsec:periods}. At the bad fibre, we have the relation
  \[ \mathbf{b}^+_{\hf} \bmod X = c_r t^{r+1} \eta_{f_r}^{\alpha} \otimes e_{-(r + 1)} \]
  where $e_n$ is the standard basis of $\Zp(n)$ and $c_r$ is a nonzero constant. Since multiplication by $t^{r+1}$ corresponds to multiplication by $\log^{[r + 1]}$ on the $\cH_{\Gamma}$ side, we conclude that
  \[ \operatorname{Col}_{\mathbf{b}_{\hf}^+ \otimes \eta_{\hg}}({}_d \kappa(\hf, \hg)) \bmod X = c_r \left\langle \mathcal{L}^{\mathrm{PR}}_{\cF^- V(\hg)^*(\psi^{-1})}({}_d \kappa(E_{\ell+2}^{\crit},\hg)), \eta_{f_k}^{\alpha} \otimes \eta_{\hg}\right\rangle.
  \]

  \begin{theorem}
   We have
   \[ \operatorname{Col}_{\mathbf{b}_{\hf}^+ \otimes \eta_{\hg}}({}_d \kappa(\hf, \hg)) = c_{\hf}(\hk) \cdot C_d(\hf,\hg, \hj) \cdot  L_p^{\hg}(\hf, \hg), \]
   where $c_{\hf}(\hk)$ is a meromorphic function on $V_1$ alone, regular and non-vanishing at all integer weights $k \ge -1$ except possibly at $k = r$, where it is regular.
  \end{theorem}

  \begin{proof}
   It follows easily from the reciprocity laws of Theorem \ref{thm:rec-law-cycles} that the quotient \[ \operatorname{Col}_{\mathbf{b}_{\hf}^+ \otimes \eta_{\hg}}({}_d \kappa(\hf, \hg)) / \left(C_d(\hf,\hg, \hj) L_p^{\hg}(\hf, \hg)\right) \] is a function of $\hk$ alone, and this ratio does not vanish at any integer $k \ge -1$ where $g_k$ is classical and cuspidal; it is equal to the fudge-factor $c_k$ defined above using the results of \ref{subsec:periods}.

   Moreover, since $L_p^{\hg}(\hf, \hg)$ is well-defined and non-zero along $\{ r \} \times V_2 \times \cW$, we conclude that $c_{\hf}(\hk)$ does not have a pole at $\hk$ (although it might have a zero there).
  \end{proof}

 \subsection{Meromorphic Eichler--Shimura}

  \begin{theorem}
   There exists an integer $n \ge 0$, and a unique isomorphism of $\cO(V_1)$-modules
   \[ \omega_{\hf}: \Dcris\left(\cF^+ D(\hf)^*(-1-\hk) \right) \cong X^{-n} \cO(V_1), \]
   whose specialisation at every $k \ge 1 \in V_1$ with $k \ne r$ is the linear functional given by pairing with the differential form $\omega_{f_k}$ assocated to the weight $k+2$ specialisation of $\hf$. For this $\omega_{\hf}$, we have
   \[ \left\langle \mathcal{L}^{\mathrm{PR}}_{\cF^{+-}}({}_d \kappa(\hf, \hg)), \omega_{\hf}^+ \otimes \eta_{\hg}\right\rangle = C_d(\hf,\hg, \hj) \cdot L_p^{\hg}(\hf, \hg).\]
  \end{theorem}

  \begin{proof}
   We simply define $\omega_{\hf}$ to be the quotient $b_{\hf}^+ / c_{\hf}$, and $n$ the order of vanishing of $c_{\hf}$ at $k = r$.
  \end{proof}

  \begin{remark}
   Note that we used the family $\hg$ in the construction of $\omega_{\hf}$; but the interpolating property relating it to the classical Eichler--Shimura isomorphism implies that it is uniquely determined by $\hf$ alone.
  \end{remark}

 \subsection{Leading terms when $c_\hf(r) = 0$}

 If $c_{\hf}(r) \ne 0$, then we have thus constructed a class in Iwasawa cohomology of $V(\hg \times \psi)^*$ whose regulator agrees with the product of Kato's Euler system for $\hg \times \psi$, and a shifted copy of the $p$-adic $L$-function for $\hg \times \tau$.

 We claim that if $c_{\hf}(r) = 0$, then in fact ${}_d \kappa(\hg, \hh)$ is divisible by $X$. If $c_{\hf}(r) = 0$, then ${}_d \kappa(E_{\ell+2}^{\crit},\hg)$ is in the Selmer group with local condition $\cF^+V(\hg)^*$. This Selmer group is 0 by \cref{prop:Selvanish}. So ${}_d \kappa(\hg, \hh) \bmod X$ would have to land in the cohomology of $V(\hg)^*_{\mathrm{sub}}$ instead; but then we are seeing the projection into $\cF^-$, not $\cF^+$, so by Kato's results again (for $\hg \times \tau$, instead of $\hg \times \psi$, this time) this is zero as well.

 So we can divide out a factor of $X$ from both ${}_d \kappa(\hf, \hg)$ and $c_{\hf}(\hk)$, and repeat the argument. Since $c_{\hf}$ is not identically 0, this must terminate after finitely many steps. Thus we have shown the following:

 \begin{proposition}
  Let $n \ge 0$ be the order of vanishing of $c_{\hf}$ at $k = r$. Then $X^{-n} {}_d \hat \kappa(\hf, \hg)$ is well-defined and non-zero modulo $X$; and this leading term projects non-trivially into the quotient $ H^1(\QQ, V(\hg)^*(\psi^{-1}) \otimes \cH_{\Gamma}(-\hj))$. Its image under the Perrin-Riou regulator is given by
  \[ c^{*}_{\hf}(r) C_d(f_\beta, \hg, \hj) \cdot \log^{[r+1]} \cdot L_p^{\hg}(E_{r + 2}(\psi, \tau), \hg), \]
  where $c^{*}_{\hf}(r) \in L^\times$.
 \end{proposition}

 We denote the resulting class by ${}_d \hat\kappa^*(f_\beta, \hg)$, and define ${}_d \kappa^*(f_\beta, \hg)$ in the same way using ${}_d \kappa(f_\beta, \hg)$ instead.. If $n = 0$, we have seen above that this class is divisible by $\log^{[r+1]}$; for $n > 0$ this is less obvious, but it follows from the proof of the next theorem:


\begin{theorem}
We have
\[ {}_d \hat\kappa^*(f_\beta, \hg) = \frac{\left(C \cdot C_d(f_\beta, \hg, \hj) \log^{[r+1]} \cdot L_p(\hg \otimes \tau, \hj - 1 - r)\right)}{L_p(\Ad \hg)} \cdot \kappa(\hg \times \psi)\]
for some nonzero constant $C \in L^\times$.
\end{theorem}

\begin{proof}
Taking $C = c^{*}_{\hf}(r)$, it follows from \cref{eq:Artin} and the previous proposition (combined with Kato's reciprocity law for $\hg$) that both of the cohomology classes we are considering have the same image under the regulator; so they are equal as cohomology classes, by \cref{prop:Selvanish}.
 \end{proof}

\begin{remark}
In terms of the class ${}_d \kappa^*(f_\beta, \hg)$, which already includes the logarithmic factor, the result takes the easier form of \[ {}_d \kappa^*(f_\beta, \hg) = \frac{\left(C \cdot C_d(f_\beta, \hg, \hj) \cdot L_p(\hg \otimes \tau, \hj - 1 - r)\right)}{L_p(\Ad \hg)} \cdot \kappa(\hg \times \psi). \]
\end{remark}

\section{Deformation of diagonal cycles}\label{sect:diag-def}

\subsection{Setup}

In this section we consider the diagonal cycles of \cite{BSV1} attached to three modular forms, or more generally to three Coleman families.

Let $f = E_{r+2}(\psi,\tau)$ stand for the Eisenstein series of weight $r+2$ and characters $(\psi,\tau)$, with $\psi \tau = \chi_f$. As before, we take its critical-slope Eisenstein $p$-stabilisation, that we denote as $E_{r+2}^{\crit}(\psi,\tau)$ or just $E_{k+2}^{\crit}$, if the choice of characters is clear from the context. Let $(g,h)$ be two modular forms of weights $(\ell+2,m+2)$, levels $(N_g,N_h)$, and nebentypes $(\chi_g,\chi_h)$. We make the self-duality assumption $\chi_f \chi_g \chi_h = 1$, and to simplify notations, suppose that $\ell \geq m$. We further fix $p$-stabilisations of $g$ and $h$, that we denote as $g_{\alpha}$ and $h_{\alpha}$, respectively. Under the non-criticality conditions already discussed, we may fix a triple of Coleman families $(\hf,\hg,\hh)$ passing through $(E_{r+2}^{\crit},g_{\alpha},h_{\alpha})$ over a triple of affinoid discs $(V_1,V_2,V_3)$. For simplicity, we may assume that both $\hg$ and $\hh$ are ordinary families, and as before, that for all integers $k \in V_1 \cap \ZZ_{\ge 0}$ with $k \ne r$, the specialisation $f_k$ is a non-critical slope cusp form.

As in \cref{es6} above, we choose a value of $\tfrac{\hk_1 + \hk_2 + \hk_3}{2}$ as a family of characters over $V_1\times V_2\times V_3$, and we say a triple of integer weights $(k, \ell, m)$ is an ``integer point'' if it is compatible with this choice of square roots.

We want to consider the diagonal class attached by the works of Darmon--Rotger \cite{DR3} and Bertolini--Seveso--Venerucci \cite{BSV1} to the triple $(\hf,\hg,\hh)$, that we denote by $\kappa(\hf,\hg,\hh)$. Recall the module $V(\hf,\hg,\hh)^* \coloneqq V(\hf)^* \hat \otimes_{\Qp} V(\hg)^* \hat \otimes_{\Qp} V(\hh)^* \otimes \mathcal H_{\Gamma}(-1-\tfrac{\hk_1 + \hk_2 + \hk_3}{2})$, defined in \cref{es6}. We define similarly a space $V^c(\hf,\hg,\hh)^*$ using $V^c(\hf)^*$ instead of $V(\hf)^*$.


\subsection{Selmer vanishing}

With the previous notations, consider the family of representations over $\cO(V_2 \times V_3)$ given by
\[
 V(\hg,\hh)_0^* := \left(V(\hg)^* \hat\otimes_{\QQ_p} V(\hh)^*\right)(-1-\tfrac{r + \hk_2 + \hk_3}{2}).
\]
Here $\tfrac{r + \hk_2 + \hk_3}{2}$ is understood as a character of $\Zp^\times$ via our choice above specialised at $\hk_1 = r$. This has a rank 1 submodule
\[ \cF^{++} V(\hg,\hh)_0^* = \left(\cF^+ V(\hg)^* \hat \otimes_{\QQ_p} \cF^+ V(\hh)^*\right) (-1-\tfrac{r + \hk_2 + \hk_3}{2}). \]

Let $n$ be an integer number, playing the role of a Tate twist (later we will take $n$ to be either $0$ or $1 +r$). We consider the two-variable $p$-adic $L$-functions $L_p^{\hg}(\hg,\hh)$ and $L_p^{\hh}(\hg,\hh)$ restricted to $s=2+\tfrac{r + \hk_2 + \hk_3}{2}-n$, which are analytic functions on $V_2 \times V_3$. We denote these functions as $L_p^{\hg}(\hg,\hh)|_n$ and $L_p^{\hh}(\hg,\hh)|_n$, respectively.

\begin{lemma}
Let $n \neq \frac{r+1}{2}$. Then the $p$-adic $L$-functions $L_p^{\hg}(\hg,\hh)|_n$ and $L_p^{\hh}(\hg,\hh)|_n$ are non-zero.
\end{lemma}

\begin{proof}
These are two-variable $p$-adic $L$-functions depending on the two-weight variable, and interpolating the cyclotomic twist corresponding to a translation of $t=\frac{r+1}{2}-n$ of the central value, which is $\frac{\ell+m+3}{2}$. If $t \geq 1$, the non-vanishing follows from the convergence of the Euler product. The case $t=\frac{1}{2}$ follows from results of Shahidi \cite[Theorem 5.2]{shahidi81} on non-vanishing of $L$-functions for $\text{GL}_n$ on the abscissa of convergence.
\end{proof}

\begin{proposition}\label{sel-van-diag}
For any integer $n \neq \frac{r+1}{2}$ and prime-to-$p$ Dirichlet character $\chi$, the ``Greenberg Selmer group''
\[ H^1_{++}(\QQ, V(\hg,\hh)_0^*(\chi)(n)) \coloneqq \operatorname{ker}\Bigg(H^1(\QQ, V(\hg,\hh)_0^*(\chi)(n)) \to \frac{H^1(\Qp, V(\hg,\hh)_0^*(\chi)(n))}{H^1(\Qp, \cF^{++}(V(\hg,\hh)_0^*(\chi)(n)))}\Bigg) \] vanishes.
\end{proposition}

\begin{proof}
We can arrange $\chi=1$ without loss of generality. We shall compare the Greenberg Selmer group above (defined by a ``codimension 3'' local condition) with the Selmer group defined by a less restrictive ``codimension 2'' local condition,
\[ H^1_{\circ +}(\QQ, V(\hg,\hh)_0^*(\chi)(n)) \coloneqq \operatorname{ker}\Bigg(H^1(\QQ, V(\hg,\hh)_0^*(\chi)(n)) \to \frac{H^1(\Qp, V(\hg,\hh)_0^*(\chi)(n))}{H^1(\Qp, \cF^{\circ +}(V(\hg,\hh)_0^*(\chi)(n)))}\Bigg), \]
where
\[ \cF^{\circ +} V(\hg,\hh)_0^* =\left(V(\hg)^* \hat \otimes_{\QQ_p} \cF^+ V(\hh)^*\right)(-1-\tfrac{r+\hk_2 + \hk_3}{2}). \]
If $(x, y)$ is any point (not necessarily classical) of $V_1 \times V_2$ at which $L_p^{\hg}(\hg,\hh)|_n$ does not vanish, then the theory of Beilinson--Flach elements shows that $H^1_{\circ +}(\QQ, V(g_x, h_y)^*(n))$ is zero, and hence \emph{a fortiori} so is $H^1_{++}(\QQ, V(g_x, h_y)^*(n))$. Hence any element of $H^1_{++}(\QQ, V(\hg,\hh)_0^*(n))$ must specialise to 0 at a Zariski-dense set of points of $V_2 \times V_3$. On the other hand, this module is contained in the full $H^1$, which is $\cO(V_2 \times V_3)$-torsion-free, by a similar argument as in the previous section. So $H^1_{++}$ is the zero module.
\end{proof}
%
%

\subsection{Families over punctured discs}

As before, we have a freeness result.
\begin{proposition}
The cohomology $H^1(\QQ, V(\hf,\hg,\hh)^*)$ is a finitely-generated module over $\mathcal O(V_1 \times V_2 \times V_3)$, and this module is $X$-torsion free, where $X \in \mathcal O(V_1)$ is a uniformizer at $r$.
\end{proposition}

\begin{proof}
This follows via the exact sequence of cohomology from the vanishing of $H^0(\QQ,V(\hf,\hg,\hh)^*/X)$, which is a consequence of specializing the families at different weights, thus excluding the option of having any $G_{\QQ}$-invariant.

Alternatively, we may see that there are no $G_{\QQ}$-invariants by establishing the stronger statement that there are no $G_{\QQ_p}$-invariants, via the same analysis of the Hodge--Tate weights as in \cite[Lemma 8.2.6]{KLZ17}.
\end{proof}

\begin{proposition}
There exists a cohomology class
\[ \kappa(\hf,\hg,\hh) \in H^1(\QQ, \frac{1}{X} V^c(\hf,\hg,\hh)^*), \]
whose fibre at any balanced integer point $(k, \ell, m)$ with $k \ne r$ is the diagonal-cycle class of \cref{es6}.
\end{proposition}

\begin{proof}
As noted in \cref{rmk:notperfect}, the construction of cohomology classes \cite[\S 8]{BSV1} does not quite work in the present setting, because the Ohta pairing $V(\hf) \times V^c(\hf)^* \to \cO(V_1)$ used in equation (82) of \emph{op.cit.} is not perfect. However, we have shown above that the Ohta pairing does induce a perfect duality between $V^c(\hf)$ and $\tfrac{1}{X}V^c(\hf)^*$ (for small enough $V_1$); and substituting this statement for the erroneous claim in \emph{op.cit.} we obtain a cohomology class valued in $\frac{1}{X} V^c(\hf, \hg, \hh)^*$ interpolating the diagonal-cycle classes for classical specialisations.
\end{proof}

\begin{remark}
For any affinoid subdomain $V_1' \subset V_1$ not containing $r$, the construction of \cite{BSV1} does apply over $V_1'$; and the restriction of our class $\kappa(\hf,\hg,\hh)$ to $V_1' \times V_2 \times V_3$ is the diagonal-cycle cohomology class $\kappa\left(\hf |_{V_1'}, \hg, \hh\right)$ of \emph{op.cit.}.
\end{remark}

\subsection{Local properties at $p$}

Consider the rank 4 submodule
\[
 \begin{aligned} \mathcal F^+_{\mathrm{bal}} D^c(\hf,\hg,\hh)^* & = (\mathcal F^+ D^c(\hf)^* \hat \otimes \mathcal F^+ D(\hg)^* \hat \otimes D(\hh)^* + \mathcal F^+ D^c(\hf)^* \hat \otimes D(\hg)^* \hat \otimes \mathcal F^+ D(\hh)^* \\ & +  D^c(\hf)^* \hat \otimes \mathcal F^+ D(\hg)^* \hat \otimes \cF^+ D(\hh)^*)(-1-\tfrac{\hk_1 + \hk_2 + \hk_3}{2}).
 \end{aligned}
\]
We also consider the quotient \[ \mathcal F^-_{\mathrm{bal}} D^c(\hf,\hg,\hh)^* = \frac{D^c(\hf,\hg,\hh)^*}{\mathcal F^+ D^c(\hf,\hg,\hh)^*}. \]
By construction, for weights in the balanced region, the submodule $\cF^+_{\mathrm{bal}}$ satisfies the Panchishkin condition (i.e.~all its Hodge--Tate weights are $\ge 1$, and those of the quotient are $\le 0$).

\begin{proposition}
The image of $\operatorname{loc}_p\left(\kappa(\hf,\hg,\hh)\right)$ in $H^1\left(\Qp, \tfrac{1}{X}\cF^{-}_{\mathrm{bal}} D^c(\hf, \hg, \hh)^*\right)$ is zero.
\end{proposition}

\begin{proof}
This follows from the fact that the Galois module is torsion free, and the specialisations away from $X = 0$ have the required vanishing property (as they are built from cohomology classes which satisfy the Bloch--Kato local condition).
\end{proof}

\subsection{Specialisation at $X=0$}

\begin{proposition}
The image of $\kappa(\hf,\hg,\hh)$ in the cohomology of the quotient
\[ \frac{\frac{1}{X} V^c(\hf,\hg,\hh)^*}{V(\hf,\hg,\hh)^*} \cong  V(\hg, \hh)^*_0 (\tau^{-1})(1+r) \]
is zero.
\end{proposition}

\begin{proof}
The image of the balanced submodule $\mathcal F^+_{\mathrm{bal}}$ in this quotient is exactly the local condition defining the Greenberg Selmer group $H^1_{++}$ considered above (with $\chi = \tau^{-1}$ and $n = 1+r$). By \cref{sel-van-diag}, the Selmer group with this local condition is zero. Hence $\kappa(\hf,\hg,\hh)$ must map to the zero class in this module.
\end{proof}

\begin{corollary}
The class $\kappa(\hf,\hg,\hh)$ lifts (uniquely) to $H^1(\QQ,V(\hf,\hg,\hh)^*)$, and thus has a well defined image in the module
\[ \hat{\kappa}(f_{\beta},\hg,\hh) \in H^1\left(\QQ, V(\hg, \hh)^*_0(\psi^{-1})\right). \]
\end{corollary}

\begin{proposition}
The class $\hat{\kappa}(f_{\beta},\hg,\hh)$ is divisible by the logarithmic distribution $\log^{[r + 1]}(\tfrac{r - \hk_2 + \hk_3}{2})$.
\end{proposition}

\begin{proof}
We identify weights with quadruples $(k,\ell,m,j)$ with $k+\ell+m=2j$. We claim that the class $\kappa(\hf,\hg,\hh)$ vanishes at $(r,\ell,m,j)$ for all $\ell,m \geq 0$ such that $(r,\ell,m)$ is balanced, i.e. $|\ell-m| \le r$ with $\ell+m+r$ even. Indeed, the specialization of $\kappa(\hf,\hg,\hh)$ at one such point factors through the image of $\sD_{U - j} \otimes \operatorname{TSym}^j$ in $\sD_{U - (r+1)} \otimes \operatorname{TSym}^{(r+1)}$, and the maps $\Pr_{\hf}^{[j]}$ and $\Pr_{\hf}^{[r+1]}$ agree on this image up to a non-zero scalar.

Since $\Pr_{\hf}^{[j]}$ for $0 \le j \le r$ do not have poles at $X = 0$, it follows that the specialisations of $\kappa(\hf,\hg,\hh)$ at triples $(r, \ell, m, \chi)$, for $|\ell-m| \le r$, $\ell+m+r$ even and $\chi$ locally-algebraic of degree $\in \{0, \dots, r\}$, interpolate the projections of the diagonal cycles to the $(E_{r+2}^{\crit}, g_\ell,h_m)$-eigenspaces in classical cohomology. Since the diagonal classes lift to $X_1(N) \times Y_1(N) \times Y_1(N)$, these projections are always 0. By Zariski-density, the class specialises to 0 everywhere in $(\{r \} \times V_2 \times V_3) \cap (|\ell-m| \leq r)$ with $\ell+m+r$ even, and the desired divisibility follows.
\end{proof}

Since the Iwasawa cohomology is torsion-free, there is a unique class \[ \kappa(f_{\beta},\hg,\hh) \in H^1(\QQ, V(\hg,\hh)_0^*(\psi^{-1})) \] such that \[ \hat{\kappa}(f_{\beta}, \hg, \hh) = \log^{[r + 1]}(\tfrac{r - \hk_2 + \hk_3}{2}) \cdot \kappa(f_{\beta},\hg,\hh). \]

\begin{proposition}
This class $\kappa(f_{\beta},\hg,\hh)$ maps to 0 in the cohomology of the rank-one quotient $\QQ_p(\psi^{-1}) \otimes \cF^- V(\hg)^* \otimes \cF^- V(\hh)^*$.
\end{proposition}

\begin{proof}
 Since $\kappa(\hf, \hg, \hh)$ lies in the balanced Selmer group, its image in $V(\hf)^* \otimes \cF^- V(\hg)^* \otimes \cF^- V(\hh)^*$ vanishes identically over $V_1 \times V_2 \times V_3$. So it is in particular zero when we specialise at $\hk = r$.
\end{proof}


\subsection{The $p$-adic $L$-function}

We assume for the remaining of this section that the tame level of the Eisenstein series is trivial. For the construction of the triple product $p$-adic $L$-function, the interpolation property also applies at $k = r$. Because of the functional equation for the Rankin $L$-function and our assumption on the tame level, there is an equality
\[
 L(E_{r+2}(\psi,\tau), g_{\ell}, h_m, 2+\frac{r+\ell+m}{2}) = L(g_{\ell}, h_m \times \psi, 2+\frac{r+\ell+m}{2})^2,
\]
where we have used that $\psi \tau \chi_g \chi_h = 1$. Observe that the restriction of $\Lp^{\bf g}(\hf,\hg,\hh)$ to the region defined by $k = r$ is uniquely determined by the interpolation property at crystalline points, and we have then an equality of $p$-adic $L$-functions \[ \Lp^{\hg}(E_{r+2}(\psi,\tau),\hg,\hh) = L_p^{\hg}(\hg, \hh \times \psi,2+\frac{r+\hk_2+\hk_3}{2}). \]


\subsection{Perrin-Riou maps}

We can relate the previous $p$-adic $L$-function to the image of $\text{loc}_p(\kappa(\hf,\hg,\hh)$ under the projection to $\mathcal F^- V(\hg)^* \otimes \mathcal F^+ V(\hh)^*$. More precisely, Perrin-Riou's regulator gives us a map
\[
  \operatorname{Col}_{\mathbf{b}_{\hf}^+ \otimes \eta_{\hg} \otimes \omega_{\hh}} = \left\langle \mathcal{L}^{\mathrm{PR}}_{\cF^{+-+}}(-), \mathbf{b}_{\hf}^+ \otimes \eta_{\hg} \otimes \omega_{\hh} \right\rangle :
  H^1(\QQ_p, \cF^{+-+} D(\hf, \hg, \hh)^*) \to \cO(V_1 \times V_2 \times V_3)
 \]
 which interpolates the Perrin-Riou regulators for $f_k \otimes g_{\ell} \times h_m$. Indeed, for $z \in \left(\cF^{+-+} D(\hf,\hg, \hh)^* \right) ^{\varphi^{-1} = 1}$, this map sends $z$ to
 \[ \langle\iota((1 - \varphi) z), \mathbf{b}_{\hf}^+ \otimes \eta_{\hg} \otimes \omega_{\hh} \rangle, \]
 where $\iota$ is now the inclusion
 \[
  \left(\cF^{+-+} D(\hf, \hg, \hh)^*\right)^{\varphi^{-1} = 0} \into
  \Dcris\left(\cF^{+}D(\hf)^*(-1-\hk_1)\right)\otimes \Dcris\left(\cF^{-}D(\hg)^*\right) \otimes \Dcris\left(\cF^{+}D(\hh)^*(-1-\hk_3) \right).
 \]
Proceeding as with Beilinson--Flach classes, we conclude that
 \[ \operatorname{Col}_{\mathbf{b}_{\hf}^+ \otimes \eta_{\hg} \otimes \omega_{\hh}}(\kappa(\hf,\hg, \hh)) \bmod X = c_r \left\langle \mathcal{L}^{\mathrm{PR}}_{\cF^{-+} V(\hg)^* \hat \otimes_{\QQ_p} V(\hh)^*(\psi^{-1})}(\kappa(E_{r+2}^{\crit},\hg,\hh)), \eta_{f_r}^{\alpha} \otimes \eta_{\hg} \otimes \omega_{\hh} \right\rangle.
 \]

The following result follows from the reciprocity law of \cite{BSV1}, with the obvious modifications to adapt it to the Coleman case, exactly as in \cite{loefflerzerbes16}.

 \begin{theorem}
  We have
  \[ \operatorname{Col}_{\mathbf{b}_{\hf}^+ \otimes \eta_{\hg} \otimes \omega_{\hh}}(\kappa(\hf, \hg, \hh)) = c(\hk) \cdot \Lp^{\hg}(\hf,\hg, \hh), \]
  where $c(\hk)$ is a meromorphic function on $V_1$, regular and non-vanishing at all integer weights $k \ge -1$ except possibly at $k=r$ itself, where it is regular.
 \end{theorem}

\begin{proof}
It follows easily from the reciprocity laws for diagonal cycles that $ \operatorname{Col}_{\mathbf{b}_{\hf}^+ \otimes \eta_{\hg} \otimes \omega_{\hh}}(\kappa(\hf,\hg, \hh)) / L_p^{\hg}(\hf, \hg, \hh)$ is a function of $\hk$ alone, and this ratio does not vanish at any integer $k \ge -1$ where $f_k$ is classical; it is equal to the fudge-factor $c_k$ defined above using \ref{subsec:periods}.

Moreover, since $L_p^{\hg}(\hf,\hg, \hh)$ is well-defined and non-zero along $\{ \ell\} \times V_2 \times V_3$, we conclude that $c(\hk)$ does not have a pole at $\hk$ (although it might have a zero there).
 \end{proof}

 \begin{remark}
 In this study we have only considered the reciprocity law for the $p$-adic $L$-function where the dominant family is not the one passing through the critical Eisenstein series point, where some subtle complications may arise. We expect to come back to this issue in forthcoming work.
 \end{remark}

\subsection{Leading terms}
If $c(r) \ne 0$, then we have thus constructed a class in the cohomology of $V(\hg \times \hh \times \psi)^*$ whose regulator agrees with that of Beilinson--Flach's Euler system for $\hg \times \hh \times \psi$.

We claim that if $c(r) = 0$, then in fact $\kappa(\hf, \hg, \hh)$ is divisible by $X$. If $c(r) = 0$, then $\kappa(E_{\ell+2}^{\crit},\hg,\hh)$ is in the Selmer group with local condition $\cF^+V(\hg)^* \hat \otimes V(\hh)^*$, which is zero following the proof of Proposition \cref{sel-van-diag} (considering now only one of the $p$-adic $L$-functions, and therefore a slightly different local condition). So $\kappa(\hf, \hg, \hh) \bmod X$ would have to land in the cohomology of $V(\hg)_{\mathrm{sub}}^* \otimes V(\hh)^*$ instead; but then we are seeing the projection into $\cF^-$, not $\cF^+$, so by the local properties of Beilinson--Flach elements again (for $\hg \times \hh \times \tau$, instead of $\hg \times \hh \times \psi$, this time) this is zero as well.

So we can divide out a factor of $X$ from both $\kappa(\hf, \hg,\hh)$ and $c(\hk)$, and repeat the argument. Since $c$ is not identically 0 this must terminate after finitely many steps.

\begin{proposition}
  Let $n \ge 0$ be the order of vanishing of $c_{\hf}$ at $k = r$. Then $X^{-n} \hat \kappa(\hf, \hg, \hh)$ is well-defined and non-zero modulo $X$. This leading term projects non-trivially into the quotient $ H^1(\QQ, V(\hg)^* \otimes V(\hh)^* (\psi^{-1})(-\hj))$. Its image under the Perrin-Riou regulator is given by
  \[ c^{*}_{\hf}(r) \cdot \log^{[r+1]} \cdot \Lp^{\hg}(E_{r + 2}(\psi, \tau), \hg, \hh), \]
  where $c^{*}_{\hf}(r) \in L^\times$.
 \end{proposition}

 We denote the resulting class by $\hat\kappa^*(f_\beta, \hg, \hh)$. If $n = 0$, we have seen above that this class is divisible by $\log^{[r + 1]}(\tfrac{r - \hk_2 + \hk_3}{2})$; for $n > 0$ this is less obvious, but proceeding as before it follows from the proof of the next theorem:

\begin{theorem}
Under the big image assumptions of \cite[\S 11]{KLZ17}, we have
\[ \hat\kappa^*(f_\beta, \hg, \hh) = \left(C \cdot \log^{[r + 1]}(\tfrac{r - \hk_2 + \hk_3}{2}) \right) \cdot {}_d \kappa (\hg, \hh \times \psi), \]
for some nonzero constant $C$ and where ${}_d \kappa (\hg, \hh \times \psi)$ is the two-variable Beilinson--Flach class indexed by the two weight variables.
\end{theorem}

\begin{proof}
The class obtained from the diagonal cycle lies in $V(\hg)^* \otimes V(\hh \times \psi)^*(-1-\tfrac{r+\hk_2+\hk_3}{2})$, so it lives in the same space as the Beilinson--Flach class for $\hj = 1 + \tfrac{r+\hk_2+\hk_3}{2}$. Then the $d$-factor is actually constant over $V_2 \times V_3$ and its value is
 \[ d^2 - d^{-(\hk_2 + \hk_3 - 2\hj)} (\varepsilon_\hf\varepsilon_\hg \psi^2)(d)^{-1}) = d^2 - d^{2 + r} (\varepsilon_\hf\varepsilon_\hg \psi^2)(d)^{-1} = d^2\left(1 - d^r \tau\psi^{-1}(d)\right).\]
Note that the ``$p$-decency'' hypothesis implies that $\tau\psi^{-1}$ must be non-trivial if $r = 0$, so we can choose $d$ such that $d^2\left(1 - d^r \tau\psi^{-1}(d)\right) \ne 0$. Hence, we may take $C = d^{-2} \left(1 - d^r \tau\psi^{-1}(d)\right)^{-1} c_{\hf}^*(r)$. From the previous proposition, together with the explicit reciprocity law for Beilinson--Flach elements, both of the cohomology classes we are considering have the same image under the regulator; so they are equal by \cref{sel-van-diag}. Note that we need to assume the big image assumptions of \cite[\S 11]{KLZ17} to assure that the Selmer group with Greenberg condition is one dimensional.
\end{proof}

As before, note that using instead $\hat\kappa^*(f_\beta, \hg, \hh)$ the result takes the simpler form \[ \kappa^*(f_\beta, \hg, \hh) = C \cdot {}_d \kappa (\hg, \hh \times \psi). \]

\section{Deformation of Heegner points}\label{sect:heeg-def}

\subsection{Setup}

We consider the Heegner point anticyclotomic Euler system of \cite{JLZ}, and keep the notations of \cref{es2} and \cref{es3}. Let $f = E_{r+2}(\psi,\tau)$ stand for the Eisenstein series of weight $r+2$ and characters $(\psi,\tau)$, with $\psi \tau = 1$. As before, let $f_{\beta}$ be its critical-slope $p$-stabilisation. Consider the unique Coleman family $\hf$ passing through $f_{\beta}$ over some affinoid disc $V_1$. We continue assuming that for all integers $k \in V_1 \cap \ZZ_{\ge 0}$ with $k \ne r$, $f_k$ is a non-critical-slope cusp form. Recall for this section the module $V^{\ac}(\hf)^*$ defined in \cref{es3}, and consider in the same way $V^{c,\ac}(\hf)^*$, replacing $V(\hf)^*$ by $V^c(\hf)^*$.

\subsection{Families over punctured discs}


\begin{proposition}
The cohomology $H^1(K,V^{\ac}(\hf)^*)$ is a finitely-generated module over $\mathcal O(V_1 \times \mathcal W)$, and this module is $X$-torsion-free, where $X \in \mathcal O(V_1)$ is a uniformizer at $r$.
\end{proposition}

\begin{proof}
This follows again via the exact sequence of cohomology from the vanishing of $H^0(\QQ,V^{\ac}(\hf)^*/X)$.
\end{proof}

\begin{theorem}
There exists a cohomology class \[ \kappa(\hf,K) \in H^1(K, \frac{1}{X} V^{c,\ac}(\hf)^*), \]
with the following interpolation property:
\begin{itemize}
\item If $(k,j)$ are integers $\ge 0$ with $k \ne r$, then we have
\[ \kappa(\hf,K)(k,j) = z_{f_k,r} \in H^1(K, V(f_k)^* \otimes \sigma^{k-j} \bar \sigma^{j}),\]
where the element $z_{f_k,r}$ is as defined in Theorem 5.3.1 of \cite{JLZ}.
   \end{itemize}
\end{theorem}

\begin{proof}
This follows from the construction of \cite{JLZ}, with the usual changes to take into account what happens in a neighbourhood of a critical-slope Eisenstein point.
\end{proof}

\subsection{Local properties at $p$}

Recall that the choice of the embedding $K \hookrightarrow \overline{\QQ_p}$ singles out one of the primes above $p$, that we have called $\fp$. The following result gives information about the vanishing of the local class at $\fp$.

\begin{proposition}
The image of $\loc_{\fp}(\kappa(\hf,K))$ in $H^1\left(K_{\fp}, \tfrac{1}{X}\cF^{-} D^c(\hf)^*\right)$ is zero.
\end{proposition}

\begin{proof}
This follows from the fact that the Iwasawa cohomology is torsion-free, and the specialisations away from $X = 0$ have the required vanishing property.
\end{proof}

\subsection{Leading terms at $X = 0$}

\begin{proposition}
The image of $\kappa(\hf,K)$ in the cohomology of the quotient
\[ \frac{\tfrac{1}{X} V^{c,\ac}(\hf)^*}{V^{\ac}(\hf)^*} \cong K_{\fp} (\tau^{-1})(1+r) \otimes \mathcal H_{\Gamma^{\ac}}(-\hj) \] is zero.
\end{proposition}
\begin{proof}
This follows from the local properties of Heegner points \cite[Proposition 6.3.2]{JLZ} (in particular, the fact that $\loc_{\fp} \kappa(\hf, K)$ factors through the anticyclotomic Iwasawa cohomology of the rank 1 submodule $\cF^+_p D(\hf)^* \subset D(\hf)^*$, as discussed in \cref{es3}).
\end{proof}

Note that this result on previous sections relied on the vanishing of particular Selmer groups (and the fact that certain Greenberg conditions were too strong). In this case, this is automatic and does not require any reciprocity law nor any result from the theory of $p$-adic $L$-functions.

\begin{corollary}
The class $\kappa(\hf,K)$ lifts (uniquely) to $H^1(K, V^{\ac}(\hf)^*)$, and thus has a well-defined image in the module \[ \hat{\kappa}(f_{\beta},K) \in H^1(K, K_{\fp}(\psi^{-1}) \otimes \mathcal H_{\Gamma^{\ac}}(-\hj)).\]
\end{corollary}

\begin{proposition}
The image of $\hat{\kappa}(f_{\beta},K)$ in the above module is divisible by the logarithm distribution $\log^{[r+1]} \in \mathcal H_{\Gamma^{\ac}}$.
\end{proposition}

\begin{proof}
This follows the same argument of \cref{div-log}, replacing the cyclotomic algebra with the anticyclotomic one.
\end{proof}

Since the Iwasawa cohomology is torsion-free, there is a unique class \[ \kappa(f_{\beta},K) \in H^1(K, K_{\fp}(\psi^{-1}) \otimes \mathcal H_{\Gamma^{\ac}}(-\hj)) \] such that \[ \hat{\kappa}(f_{\beta},K) = \log^{[r + 1]} \cdot \kappa(f_{\beta},K). \]


\subsection{The $p$-adic $L$-function}

Recall the anticyclotomic $p$-adic $L$-function $L_{\fp}^\mathrm{BDP}(\hf)$, that was introduced in \cref{es3} as a function over $V_1 \times \cW$. For this construction, the interpolation property also works at $k = r$, and the complex $L$-functions factors as \[ L(E_{r+2}(\psi,\tau)/K \times \chi_{\ac}^j,1) = L(\psi/K \times \sigma^{r+2+j} \bar{\sigma}^{-j},1) \cdot L(\tau/K \times \sigma^{j+1} \bar{\sigma}^{-j-r-1},1) = L(\psi/K \times \sigma^{r+2+j} \bar{\sigma}^{-j},1)^2. \] Note that we have used that $L(\tau/K \times \sigma^{j+1} \bar{\sigma}^{-j-r-1},1)=L(\psi/K \times \sigma^{r+2+j} \bar{\sigma}^{-j},1)$, which follows from the functional equation together with the condition that $\psi \tau = 1$. This automatically gives an equality of $p$-adic $L$-functions \[ L_{\fp}^\mathrm{BDP}(E_{r+2}(\psi,\tau))(\chi_{\ac}^j) = L_{\fp}^\mathrm{Katz}(\psi)(\sigma^{r+2+j} \bar{\sigma}^{-j}). \] (Alternatively, it directly follows from the construction of \cite{BDP13} that both functions agree.)

\subsection{Perrin-Riou maps}

We want to relate the $p$-adic $L$-function $L_{\fp}^\mathrm{BDP}(\hf)$ to the image of $\text{loc}_{\fp}(\kappa(\hf,K))$. This factors through the natural map \[ H^1(K_{\fp}, D^{\ac}(\hf)^*) \rightarrow H^1(K_{\fp}, \mathcal F^+ D^{\ac}(\hf)^*). \]
Perrin-Riou's regulator gives a map
\[
  \operatorname{Col}_{\mathbf{b}_{\hf}^+} = \left\langle \mathcal{L}^{\mathrm{PR}}_{\cF^+V(\hf)^*}(-), \mathbf{b}_{\hf}^+ \right\rangle :
  H^1(K_{\fp}, \cF^{+} D^{\ac}(\hf)^*) \to \cO(V_1 \times \cW)
\]
which interpolates the Perrin-Riou regulators for $f_k$. More precisely, for $z \in \left(\cF^{+} D^{\ac}(\hf)^*\right)^{\varphi^{-1} = 1}$, this map sends $z$ to
 \[ \langle\iota((1 - \varphi) z), \mathbf{b}_{\hf}^+ \rangle, \]
 where $\iota$ is the inclusion
 \[
  \left(\cF^{+} D^{\ac}(\hf)^*\right)^{\varphi^{-1} = 0} \into \left(\cF^{+} D^{\ac}(\hf)^*[1/t]\right)^{\varphi^{-1} = 0}
  = \Dcris\left(\cF^{+}D^{\ac}(\hf)^*(-1-\hk)\right)\otimes \cH_{\Gamma^{\ac}}.
 \]

Since multiplication by $t^{r+1}$ corresponds to multiplication by $\log^{[r + 1]}$ on the $\cH_{\Gamma^{\ac}}$ side, we conclude that
 \[ \operatorname{Col}_{\mathbf{b}_{\hf}^+}(\kappa(\hf,K)) \bmod X = c(\hk) \left\langle \mathcal{L}^{\mathrm{PR}}_{\psi^{-1}} (\kappa(\hf,K)), \eta_{f_k}^{\alpha}\right\rangle.
 \]

\begin{theorem}
We have
\[ \operatorname{Col}_{\mathbf{b}_{\hf}^+}(\kappa(\hf,K)) = c(\hk) \cdot L_{\fp}^\mathrm{BDP}(\hf), \]
where $c(\hk)$ is a meromorphic function on $V_1$ alone, regular and non-vanishing at all integer weights $k \ge -1$ except possibly at $k=r$ itself, where it is regular.
 \end{theorem}

\begin{proof}
It follows from the reciprocity laws for Heegner points that the quotient $ \operatorname{Col}_{\mathbf{b}_{\hf}^+}(\kappa(\hf,K)) / L_p(\hf, K)$ is a function of $\hk$ alone, and this ratio does not vanish at any integer $k \ge -1$ where $f_k$ is classical; it is equal to the fudge-factor $c_k$ defined above. Since $L_{\fp}^\mathrm{BDP}(\hf)$ is well-defined and non-zero along $\{ r \} \times \cW$, we conclude that $c(\hk)$ does not have a pole at $\hk$.
\end{proof}

\subsection{Leading terms}
If $c(r) \ne 0$, then we have thus constructed a class in Iwasawa cohomology of $V(\psi)^*$ whose regulator agrees with the Euler system of elliptic units. If $c(r) = 0$, then in fact $\kappa(\hf,K)$ is divisible by $X$, so we can divide out a factor of $X$ from both $\kappa(\hf,K)$ and $c(\hk)$, and repeat the argument. Since $c$ is not identically 0 this must terminate after finitely many steps.

\begin{proposition}
  Let $n \ge 0$ be the order of vanishing of $c_{\hf}$ at $k = r$. Then $X^{-n} \hat \kappa(\hf, K)$ is well-defined and non-zero modulo $X$; and this leading term projects non-trivially into the quotient $ H^1(\QQ, K_{\fp}(\psi^{-1}) \otimes \cH_{\Gamma^{\ac}})$. Its image under the Perrin-Riou regulator is given by
  \[ c^{*}_{\hf}(r) \cdot \log^{[r+1]} \cdot L_{\fp}^\mathrm{BDP}(E_{r + 2}(\psi, \tau)), \]
  where $c^{*}_{\hf}(r) \in L^\times$.
 \end{proposition}

 We denote the resulting class by $\hat\kappa^*(f_\beta, K)$. If $n = 0$, we have seen above that this class is divisible by $\log^{[r+1]}$; for $n > 0$ this is less obvious, but it follows from the proof of the next theorem:

\begin{theorem}
We have
\[ \hat\kappa^*(f_\beta, K) = \left(C \cdot \log^{[r+1]} \cdot (-1)^{\frac{r}{2}+\hj} \right) \cdot \kappa(\psi,K)(\sigma^{-1-r-\hj} \bar{\sigma}^{1+\hj}), \]
for some nonzero constant $C$, and where $\kappa(\psi,K)(\sigma^{-1-r-\hj} \bar{\sigma}^{1+\hj})$ is the specialization at $\sigma^{-1-r-\hj} \bar{\sigma}^{1+\hj}$ of the system of elliptic units defined in \cref{es2}.
 \end{theorem}

\begin{proof}
Take as before $C = c_{\hf}^*(r)$. This follows from the previous proposition, together with the explicit reciprocity law for elliptic units, that both of the cohomology classes we are considering have the same image under the regulator, and so they must be equal.
 \end{proof}

\let\MR\undefined
\newlength{\bibitemsep}
\setlength{\bibitemsep}{0.75ex plus 0.05ex minus 0.05ex}
\newlength{\bibparskip}
\setlength{\bibparskip}{0pt}
\let\oldthebibliography\thebibliography
\renewcommand\thebibliography[1]{%
 \oldthebibliography{#1}%
 \setlength{\parskip}{\bibparskip}%
 \setlength{\itemsep}{\bibitemsep}%
}

 \bibliographystyle{../amsalphaurl}
 \bibliography{../references}

\end{document}